\documentclass[12pt,a4paper]{amsart}
\usepackage{coolstyle}
\usepackage{multicol}
\usepackage{subcaption}

\begin{document}
\raggedbottom

\title{Taming Wild Knots with Mosaics}
\author[M. ~Deng]{Mary Y. Deng}
\email[M. ~Deng]{marydeng@uw.edu}
\author[A. K. ~Henrich]{Allison K. Henrich}
\email[A. K. ~Henrich]{henricha@seattleu.edu}
\author[S. ~Kawano]{Sean H. Kawano}
\email[S. ~Kawano]{shkawano@uw.edu}
\author[A. R. ~Tawfeek]{Andrew R. Tawfeek}
\email[A. R. ~Tawfeek]{atawfeek@uw.edu}

\begin{abstract}
    Wild knots—knots with infinite knotting behavior—have resisted traditional methods of knot classification, making them more of a curiosity in topology than a subject of sustained investigation. In this paper, we present a new way to investigate these objects. We extend Lomonaco and Kauffman's knot mosaic theory to represent a substantial subclass of wild knots that have isolated wild points. Our mosaics consist of infinite rooted trees with mosaics assigned to vertices and embedding functions governing connections. In developing this framework, we also introduce a notion of mosaic tangles as well as mosaic rigid vertex spatial graphs of which mosaic singular knots are a special case. 
\end{abstract}

\maketitle
\setcounter{tocdepth}{2}
\makeatletter
\def\l@subsection{\@tocline{2}{0pt}{2.5pc}{5pc}{}}
\makeatother
\tableofcontents

\section{Introduction}



Wild knots---knots with infinite knotting behavior---are a curiosity in knot theory. The vast majority of knot theorists are concerned with studying {\bf (tame) knots}, i.e., those knots that can be represented as finite closed polygonal chains. In this setting, knot equivalence is described in terms of finite sequences of Reidemeister moves. (See e.g., \cite{adams2004knot} or \cite{encyclopedia} for more on (tame) knot theory fundamentals.)  

Though studied by few, wild knots exhibit fascinating behavior. Take, for example, the wild slip knot, pictured in Figure \ref{slip}. While this knot---referred to by Ralph Fox in 1949 as a ``remarkable simple closed curve'' in \cite{fox1949remarkable}---looks like it can be undone by pulling out successive loops, it cannot be simplified. (According to Fox, it is only {\em almost} unknotted.) One might wonder if {\em any} knot requiring an infinite number of Reidemeister moves can be simplified. In 2021, Forest Kobayashi \cite{forest} justified the application of certain infinite collections of Reidemeister moves on wild knots. 

\begin{figure}[h]
\centering
\includegraphics[width=.5\linewidth]{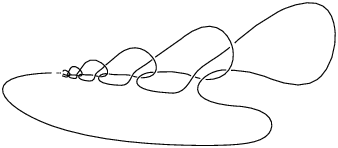} 
\caption{A wild slip knot}\label{slip}
\end{figure}

Because wild knots are difficult to classify, our aim is to represent a substantial subclass of them using a more discrete, combinatorial structure: knot mosaic tiles. In this paper, we develop and explore a new theory of wild mosaic knots.


\subsection{History and fundamentals of wild knots}

Since wild knots were defined by Fox and Artin in 1948 \cite{fox1948some, fox1949remarkable}, there have been some efforts devoted to classifying various subclasses of wild knots:  from earlier work in the 1960s and '70s \cite{HarFox,Lom,DoHo, McP} through more recent work in the 2000s \cite{Fri, JuLa, Nanyes, forest}. Also, higher dimensional wildly knotted spaces have been addressed~\cite{Fri1,Fri2}, and replacing $S^1$ by the Cantor set, one obtains the related study of the so-called {\em wild Cantor sets} whose classification has also been of interest because of their importance in dynamical systems~\cite{GRWZ}.

From a descriptive set theory perspective \cite{hjorth2000classification}, the classification of tame knots is trivial because there are only countably many of them (up to equivalence). In contrast, the following was shown by Kulikov:

\begin{theorem}[\cite{kulikov2017non}]
    The equivalence of wild knots is not classifiable by countable structures.
\end{theorem}

\vspace{-.1in}

Before we get too far, let us introduce some of the classical formal definitions we'll need. We'll later adapt these topological notions to a more combinatorial framework.

\begin{definition}[\cite{fox1949remarkable, mcpherson1970wild, p-index}]\label{wilddef} We define {\em knot}, {\em link}, {\em knot diagram}, {\em Reidemeister moves}, {\em tame knot}, {\em wild knot}, {\em locally tame}, {\em wild point}, and {\em p-index} as follows. 

\begin{enumerate}
\item A {\bf knot} is a simple closed curve embedded in three-dimensional space (either $\mathbb{R}^3$ or $S^3$).
\item A {\bf link} is a finite collection of non-intersecting knots.
\item A {\bf knot diagram} is a regular projection of a knot\footnote{Throughout the work, ``knot'' may refer to either a knot or a link.} onto $\mathbb{R}^2$ where we retain the under- and over-crossing information at each double point. 
\item A {\bf tame knot} is a knot that can be approximated by (is isotopic to) a finite polygonal chain. Equivalently, we can define a tame knot as a piecewise-linear curve in $\mathbb{R}^3$ or $S^3$ consisting of a finite collection of line segments. (Most topologists refer to tame knots simply as ``knots.'') 
\item Two tame knots are equivalent if their diagrams are related by a finite sequence of {\bf Reidemeister moves}, shown in Fig. \ref{reidMoves}.
\item A {\bf wild knot} is any knot that is not tame.
\item Let $q$ be a point on a knot $K$ in a space $S$ (where $S$ is either $\mathbb{R}^3$ or $S^3$). Knot $K$ is {\bf locally tame} at $q$ if there exists a neighborhood, $N_q$, of $q$ in $S$ such that the intersection of $K$ with the closure of $N_q$ (picture a---possibly smooshed---closed ball around point $q$) can be approximated by a finite piecewise-linear curve. 
\item A point $q$ of a knot $K$ in a space $S$ is a {\bf wild point} if it is not locally tame.

\item The {\bf penetration index} or {\bf p-index} of a point $q$ in a knot $K$ is the smallest integer $n$ such that there are arbitrarily small neighborhoods of $q$, each meeting $K$ on its boundary in $n$ points. (Note: For tame knots, $K$, the p-index of any point $q\in K$ must be 2.)
\end{enumerate}
\end{definition}

\begin{figure}[H] 
\includegraphics[width=0.7\linewidth]{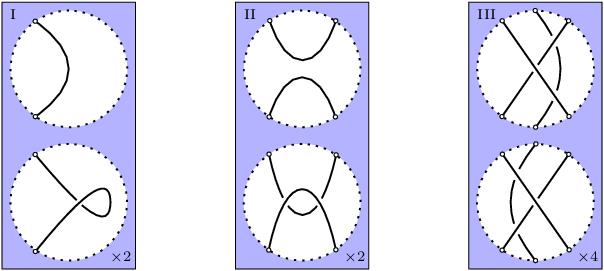}
\caption{Representatives of the three classical Reidemeister moves}\label{reidMoves}
\end{figure}

\section{Background}

In this section, we review some additional important definitions and theorems, in the realms of tangles, rigid vertex spatial graphs, and knot mosaics. We will need these to develop our wild mosaic knot framework later. 

\subsection{Tangles}

A \textbf{$k$-tangle} is a proper embedding of a disjoint union of $k$ arcs and $m$ circles into a $3$-ball, $B^3$, such that the endpoints of these arcs are glued along the ball's boundary at $2k$ marked points \cite{encyclopedia}. Knot theorists are particularly interested in $2$-tangles since they can be closed to form knots and links in a straightforward way. Traditionally, as defined by Conway in \cite{tangles}, we picture a $2$-tangle as a portion of a knot diagram that lies within a circle, where the four points at which the strands cross the circle occur in the four compass directions of NW, NE, SW, and SE. 

Two tangles are \textit{equivalent} if we can deform one into the other through a sequence of Reidemeister moves and planar isotopies while keeping the endpoints of the arcs in the tangle fixed.
 
\begin{figure}[H] \label{tangle-example}
\includegraphics[width=.65\linewidth]{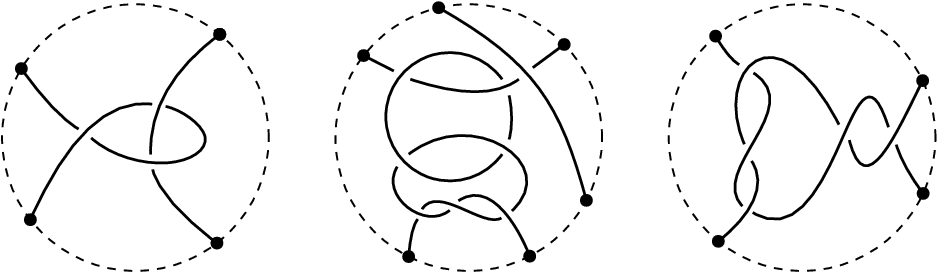}
\caption{Examples of various tangles.}
\end{figure}

\subsection{Rigid vertex spatial graphs}

Next, we review the notions of a spatial graph and rigid vertex spatial graph.

A {\bf spatial graph} is a one-dimensional CW complex consisting of finitely many \textbf{vertices} (zero-dimensional) and finitely many (one-dimensional) \textbf{edges} (including \textbf{loops}), tamely embedded in $\mathbb{R}^3$ or $S^3$. Each edge is homeomorphic to a closed interval whose boundary is mapped to either two distinct vertices in the graph or a single vertex, in the case of a loop.

In \cite{kauffman-graphs}, Kauffman presented spatial graphs as planar diagrams of vertices and edges such that each vertex of the graph locally has a collection of rays emanating from it (representing the edges that have that vertex as a boundary). When edges intersect in the plane, their over/under crossing information is given, indicating how the graph is embedded in space. By introducing new Reidemeister moves involving vertices, spatial graphs can be studied from the perspective of knot theory.

Among the types of spatial graphs Kauffman considered, he specialized to $4$-valent graphs with \textbf{rigid vertices}, which we will denote as $\rv{4}$ graphs. Kauffman describes rigid vertices as flat disks, each having (four) flexible strings emanating from it, with string endpoints attached to the disk in a distinct cyclic order. These may, equivalently, be viewed as {\bf singular knots}, i.e. knots where a finite number of transversal self-intersections are allowed. 

Rigid vertex spatial graph equivalence, or (piecewise linear) {\bf ambient isotopy}, is described via diagrams using an extended set of Reidemeister moves that captures how these rigid vertices can twist in space and interact with  other edges. Kauffman proves that these moves, shown in Figure \ref{rv4-isotopy} as moves I-IV and V*, capture the notion of ambient isotopy in space for $\rv{4}$ graphs.


\begin{theorem}[\cite{kauffman-graphs}]
    Piecewise linear (PL) ambient isotopy of embedded $\rv{4}$ graphs is generated by the moves I - IV, V* shown in Figure \ref{rv4-isotopy}. That is, if two embedded $\rv{4}$ graphs are ambient isotopic, then any two diagrams of them are related by a finite sequence of these moves. 
\end{theorem}

\begin{figure}[H] 
\includegraphics[width=0.7\linewidth]{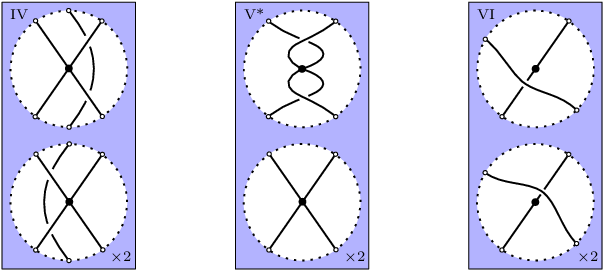}
\caption{Reidemeister moves IV, V* for $\rv{4}$ graphs and the additional move of VI for $\rv{2,4}$ graphs.}\label{rv4-isotopy}
\end{figure}

\subsection{$\rv{2,4}$ graphs}

Throughout this paper, we will work with objects that can be viewed as rigid vertex spatial graphs having degree $4$ rigid vertices, but we must also allow for degree $2$ vertices for reasons that will become clear. We refer to rigid vertex spatial graphs in which all vertices have degree $2$ or $4$ as {\bf (2,4)-rigid vertex spatial graphs}, or {\bf $\rv{2,4}$ graphs}. It is straightforward to expand $\rv{4}$ graph ambient isotopy to describe isotopies involving degree $2$ vertices by introducing one additional Reidemeister move type, pictured in its PL form in Figure \ref{reid6-fig}.

\begin{figure}[H] 
\includegraphics[width=.5\linewidth]{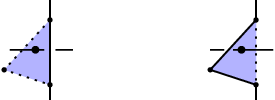} 
\caption{An illustration of Reidemeister move VI and its depiction as an elementary combinatorial isotopy.}\label{reid6-fig}
\end{figure}

Move VI, along with the moves I-IV and V* introduced in the previous section, generate PL ambient isotopy of embedded $\rv{2,4}$ graphs. The proof of the $\rv{4}$ case was already demonstrated by Kauffman \cite{kauffman-graphs}, so it remains to show that move VI is an elementary combinatorial isotopy. Note: it was proven in \cite{graeub1950semilinearen} that PL ambient isotopy and combinatorial isotopy are equivalent for PL links in three-dimensional space, so this suffices. 

To show something is an \textbf{elementary combinatorial isotopy} amounts to showing that we may embed a triangle into $\mathbb{R}^3$ piecewise linearly such that its interior is disjoint from the graph (in our case, $\rv{2,4}$) embedding, and the triangle shares either one or two edges with the embedded graph. We then replace the shared edges of the triangle with the unshared edges, and the result is a new PL embedding of the same graph with a different set of selected points. Since this is how we've depicted our new Reidemeister move in Figure \ref{reid6-fig}, we see that move VI is an elementary combinatorial isotopy. Moreover, any other elementary combinatorial isotopy involving a degree $2$ vertex can be achieved by a combination of moves I - IV, V*, and VI. So, we have the following.

\begin{theorem}
    If any two embedded $\rv{2,4}$ graphs are ambient isotopic, then any two of their planar diagrams are related by a finite sequence of the Reidemeister moves I - IV, V*, and VI.
\end{theorem}

We proceed by translating $\rv{2,4}$ graphs and their moves into a mosaic setting.

\subsection{Knot mosaics}

Knot mosaic diagrams were first introduced by Kauffman and Lomonaco in 2008 \cite{lomonaco2008quantum} towards the development of quantum knot systems. These are $n \times n$ grids of \textbf{suitably connected} knot mosaic tiles, in the sense that each connection point of a tile touches a connection point of a contiguous tile. 

\begin{definition}[Mosaic, Knot Mosaic] \label{mosaicdef}
    Denote by $\mathbb{T}=\{T_0,\dots,T_{10}\}$ the standard set of 11 tiles, shown in Figure \ref{tile-fig}. For a positive integer $n$, we define the following:
    \begin{itemize}
        \item An {\bf $n$-mosaic} $M$ is an $n \times n$ matrix with entries from $\mathbb{T}$ where $n$ is the dimension of the mosaic. The set of $n$-mosaics is denoted by $\MM^{(n)}$.
        \item A {\bf knot $n$-mosaic} is an $n$-mosaic in which all tiles are suitably connected. We let $\KK^{(n)}$ denote the subset of $\MM^{(n)}$ of all knot $n$-mosaics. We assume that connection points are {\em not} allowed along the boundary of a knot mosaic.
    \end{itemize}
\end{definition}

Later, in section \ref{tangmossec}, we will define the notion of a {\bf tangle mosaic} where connection points {\em are} allowed on the boundary.


\begin{figure}[h] 
\centering
\includegraphics[width=.9\linewidth]{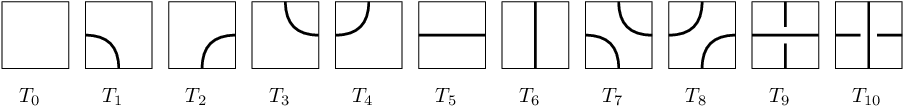} 
\caption{The eleven standard mosaic tiles.}\label{tile-fig}
\end{figure}

Just as with ordinary knot diagrams, there is a set of Reidemeister-type moves that relate mosaics representing equivalent topological knots. This mosaic-based set of moves---shown in Figures \ref{mosaicpl}, \ref{mosaic12}, and \ref{mosaic3}---is more extensive than the set of classical Reidemeister moves since planar isotopies must be exhaustively described in the mosaic setting. 

\begin{figure}
\centering
\includegraphics[width=\linewidth]{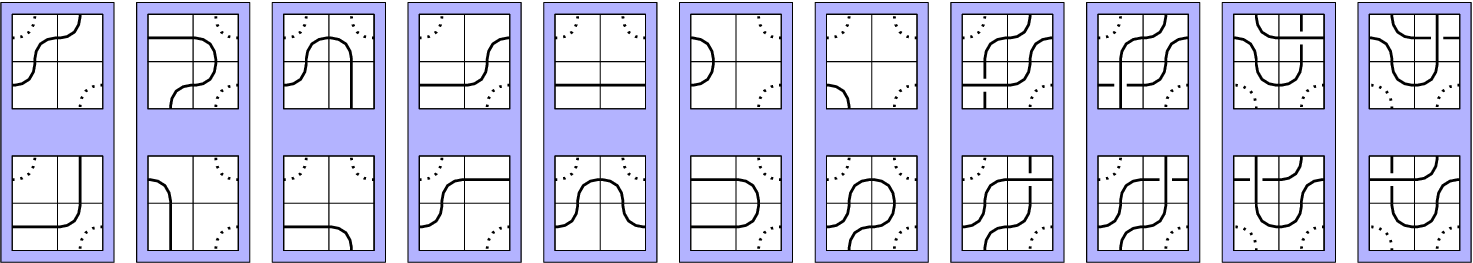} 
\caption{The eleven mosaic planar isotopy moves.}
\label{mosaicpl}
\end{figure}

\begin{figure}
\centering
\includegraphics[width=.2\linewidth]{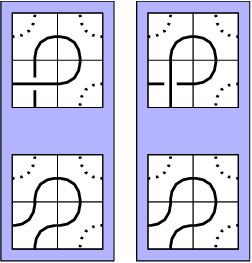} \qquad \qquad \includegraphics[width=.415\linewidth]{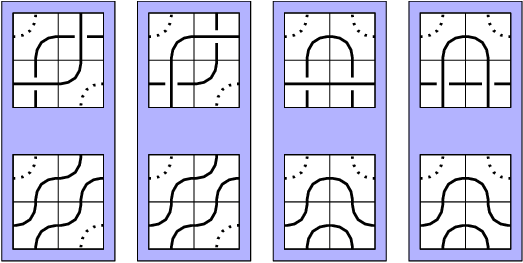} 
\caption{The mosaic Reidemeister $1$ moves and Reidemeister $2$ moves.}
\label{mosaic12}
\end{figure}

\begin{figure}
\centering
\includegraphics[width=.9\linewidth]{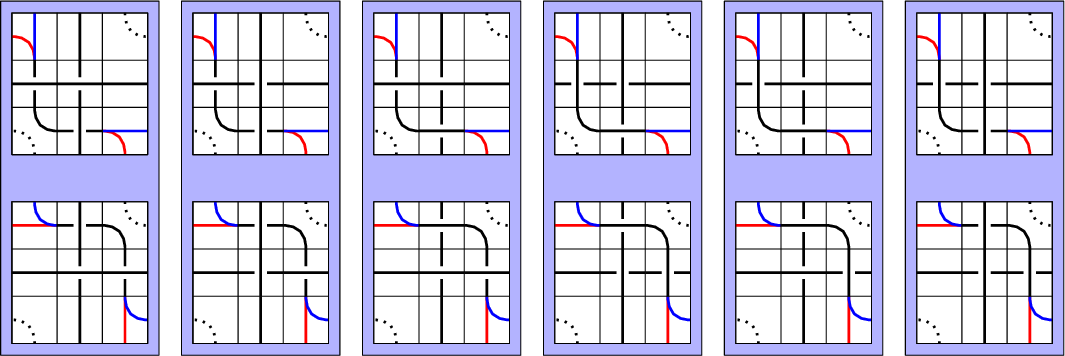} 
\caption{The mosaic Reidemeister $3$ moves. The dual-colored tiles denote two possibilities and are in sync with their isotopic counterparts.}
\label{mosaic3}
\end{figure}

Moreover, natural operations that expand or contract the size of the mosaic must be included in the list of equivalences for knot mosaics. To be precise, a knot $n$-mosaic is considered to be equivalent to a knot $m$-mosaic if they are relatable via a sequence of mosaic Reidemeister moves and {\em mosaic injections}.

\begin{definition}\label{inject}
Per \cite{lomonaco2008quantum}, a {\bf mosaic injection} $\iota: \KK^{(n)}\rightarrow \KK^{(n+1)}$ sends a knot mosaic $M^{(n)}$ to $M^{(n+1)}$, where 

\begin{equation*}
M^{(n+1)}_{ij} = 
\left\{
    \begin{array}{lr}
        M^{(n)}_{ij}, & \text{if } 1\leq i,j\leq  n\\
        \zero, & \text{otherwise.} 
    \end{array}
    \right.
\end{equation*}
\end{definition}

 Every tame knot admits a mosaic representative, as demonstrated in \cite{lomonaco2008quantum}. But Kauffman and Lomonaco went further, conjecturing that their set of knot mosaic moves completely captures topological knot equivalence. In 2014, Kuriya and Shehab proved it. See \cite{kuriya2014lomonaco}.

\begin{theorem}[Lomonaco-Kauffman Conjecture, \cite{kuriya2014lomonaco}] \label{lc-conj}
    Let $k_1$ and $k_2$ be two tame knots (or links), and let $K_1$ and $K_2$ be two arbitrarily chosen mosaic representatives of $k_1$ and $k_2$, respectively. Then $k_1$ and $k_2$ are of the same knot type if and only if the representative mosaics $K_1$ and $K_2$ are related by a finite sequence of mosaic Reidemeister moves and injections.
\end{theorem}

Once mosaics were established as an alternative representation of topological knots and links, researchers became interested in a variety of questions about them. Notably, several early teams of researchers studied combinatorial invariants including the {\bf mosaic number} of a knot, i.e. the minimum number $n$ such that the knot can be represented as a knot $n$-mosaic. See, for instance, \cite{02d_Dye2014, 02d_Ludwig2018, 02d_Ludwig2013, 02d_Lee2014-3-5, 02d_Lee2014-4-5}. In addition to combinatorial studies, geometric ideas related to mosaics have been explored in works such as \cite{carlisle2013upper, ganzell2020virtual}. We will leave these fascinating avenues of study here and move on to expanding the notion of a knot mosaic to tangles.


\section{Expanding Mosaics} 

\subsection{Tangle Mosaics}\label{tangmossec}

\begin{definition}[Tangle Mosaic] For a positive integer $n$, a {\bf tangle $n$-mosaic} $L$ is an $n$-mosaic in which all tiles are suitably connected and connection points {\em are} allowed along the boundary of the square. We let $\LL^{(n)}$ denote the subset of all tangle $n$-mosaics, and $\LL = \bigcup_{n \in \NN} \LL^{(n)}.$ Adorably, $\KK^{(n)}\subset\LL^{(n)}\subset\MM^{(n)}$.\end{definition}

Just as with ordinary knot mosaic injections (see Definition \ref{inject}), we have a notion of mosaic injections for tangle mosaics. Figure \ref{injectM} illustrates an example of this injection.

\begin{definition}\label{tangle-injections}
A {\bf tangle injection}$ \iota: \LL^{(n)}\rightarrow \LL^{(n+1)}$ sends a tangle mosaic $L^{(n)}$ to $L^{(n+1)}$, where 

\begin{equation*}
L^{(n+1)}_{ij} = 
\left\{
    \begin{array}{lr}
        L^{(n)}_{ij}, & \text{if } 1\leq i,j\leq  n\\
        \\
        \five, & j = n+1 \text{ and } L^{(n)}_{in}\in\{\two,\three,\five, \seven, \eight, \nine, \ten\}\\
        \\
        \six, & i = n+1 \text{ and } L^{(n)}_{nj}\in\{\one,\two,\six,\seven, \eight, \nine, \ten\}\\
        \\
        \zero , & \text{otherwise.}
    \end{array}
    \right.
\end{equation*}
\end{definition}

\begin{figure}
\centering
\includegraphics[width=0.6\textwidth]{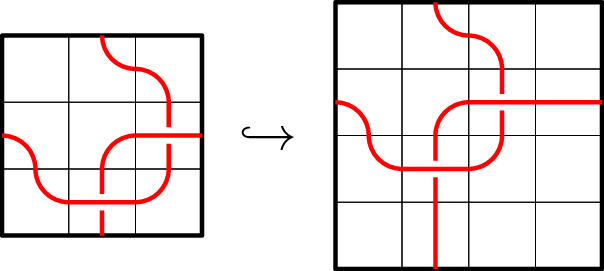}
\caption{An injection of a tangle 3-mosaic into a tangle 4-mosaic.}\label{injectM}
\end{figure}

 \begin{proposition}[Representability of Tame Tangles] \label{tangle-rep}
     Every tame $m$-tangle $\tau$ admits a tangle mosaic representative $L \in \mathbb{L}^{(n)}$ for some $n \in \mathbb{N}$, with $2m$ connection points along the boundary of $L$.
\end{proposition}

\begin{proof}
Since $\tau$ is tame, it admits a regular PL diagram in a disk whose $2m$ endpoints lie on the boundary of the disk. Embed the disk in a square and refine by a sufficiently fine square grid so that each grid cell contains at most one local feature: a straight strand segment, a turn, or a single crossing. After a small planar isotopy, the boundary endpoints occur at midpoints of boundary tile edges. Assigning to each grid cell the corresponding tile from $T_0,\ldots,T_{10}$ gives a suitably connected tangle mosaic $L \in \LL^{(n)}$ with exactly the prescribed $2m$ boundary connection points.
\end{proof}

\begin{remark}
Conversely, every mosaic tangle determines a well-defined topological tangle equivalence class. The geometric realization of a mosaic $L \in \mathbb{L}^{(n)}$ is a tangle in $[0,n]^2 \times I$ obtained by interpreting each tile as its standard local picture, and the usual mosaic Reidemeister moves are realized by ambient isotopies supported in the corresponding tile sub-cubes and fixed on the boundary. Thus the realization map descends to a well-defined map from mosaic tangle equivalence classes to topological tangle equivalence classes. Surjectivity of this map onto tame tangles is the content of the construction above; the (currently conjectural, see Conjecture~\ref{conj:lk-tangles} below) Kuriya--Shehab-type statement for tangles is the assertion that this map is also \emph{injective}.
\end{remark}


\begin{conjecture}[Topological Equivalence Conjecture for Tangles]\label{conj:lk-tangles}
The statement of Theorem \ref{lc-conj}, adapted for tangle mosaics, holds.

Precisely, let $t_1$ and $t_2$ be two tame tangles, and let $T_1$ and $T_2$ be two arbitrarily chosen mosaic representatives of $t_1$ and $t_2$, respectively. Then $t_1$ and $t_2$ are of the same tangle type if and only if, after applying finitely many tangle injections to place the representatives in a common dimension if necessary, the resulting mosaics are related by a finite sequence of mosaic Reidemeister moves.
\end{conjecture}

\subsection{Mosaics from $\rv{2,4}$ graphs}

In order to later depict wild knots, we would like to expand the tile-set $\mathbb{T}$ of mosaics to allow for one additional tile, denoted $T_\infty = \blacksquare$. This tile will play the role of the rigid vertex of an $\rv{2,4}$ graph. As such, this square may have either $2$ or $4$ connection points. For simplicity, we only allow {\em transverse} $2$-connections, i.e. connections are either on the west and east sides or the north and south sides of the tile. We consider a $T_\infty$ tile to be {\em suitably connected} if it meets these criteria.

\begin{definition}[$\rv{2,4}$ Mosaics]\label{rv24def}
    For a positive integer $n$, we define the following:
    \begin{itemize}
        \item An \textbf{$\rv{2,4}$ $n$-mosaic} is an $n \times n$ matrix with entries from $\mathbb{T} \cup \{ T_\infty\}$. We let $\mathbb{M}_{\text{RV}}^{(n)}$ denote the set of $\rv{2,4}$ $n$-mosaics and $\mathbb{M}_{\text{RV}}=\cup_{n\in\mathbb{N}}\mathbb{M}^{(n)}_{\text{RV}}$.
        \item An \textbf{$\rv{2,4}$ knot $n$-mosaic} is a $\rv{2,4}$ $n$-mosaic in which all tiles are suitably connected to their adjacent tiles, no two $\blacksquare$ tiles are adjacent, no $\blacksquare$ tiles are along the boundary, and no connection points are on the boundary of the $n\times n$ square. We let $\mathbb{K}^{(n)}_{\text{RV}}$  denote the set of $\rv{2,4}$ knot $n$-mosaics and $\mathbb{K}_{\text{RV}}=\cup_{n\in\mathbb{N}}\mathbb{K}^{(n)}_{\text{RV}}$. 
         \item An \textbf{$\rv{2,4}$ tangle $n$-mosaic} is a $\rv{2,4}$ $n$-mosaic in which all tiles are suitably connected to their adjacent tiles and no two $\blacksquare$ tiles are adjacent, $\blacksquare$ tiles along the boundary of the mosaic are prohibited, {\em but connection points (from tiles in $\mathbb T$) along the boundary are allowed}. We let $\mathbb{L}^{(n)}_{\text{RV}}$ denote the set of $\rv{2,4}$ tangle $n$-mosaics and $\mathbb{L}_{\text{RV}}=\cup_{n\in\mathbb{N}}\mathbb{L}^{(n)}_{\text{RV}}$. Note: An {\bf $m$-tangle} in this setting refers to an element of $\mathbb{L}^{(n)}_{\text{RV}}$ with exactly $2m$ connection points along its boundary.
    \end{itemize}
\end{definition}

Note that this means $\mathbb K_{\text{RV}}^{(n)} \subset \mathbb L_{\text{RV}}^{(n)} \subset \mathbb M_{\text{RV}}^{(n)}$, similar to before.

\begin{remark}\label{rv-tangle-injections}
     Because we prohibit $\blacksquare$ tiles along the boundary of $\rv{2,4}$ tangles, the tangle injection from Definition \ref{tangle-injections} can be seen as a function on $\rv{2,4}$ tangle mosaics $\iota: \LL_{\text{RV}}^{(n)}\rightarrow \LL_{\text{RV}}^{(n+1)}$, without any ambiguity.
\end{remark}

Since we will be working most often with $\rv{2,4}$ tangle mosaics, we provide an example in Figure \ref{L-ex}.

\begin{figure}[H]
\includegraphics[width=.23\linewidth]{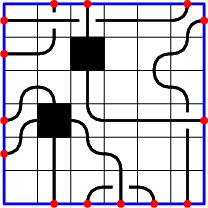}
\caption{A typical $\rv{2,4}$ tangle mosaic -- in this case, a $7$-tangle.}\label{L-ex}
\end{figure}

\begin{figure}
\centering
\includegraphics[scale=0.6]{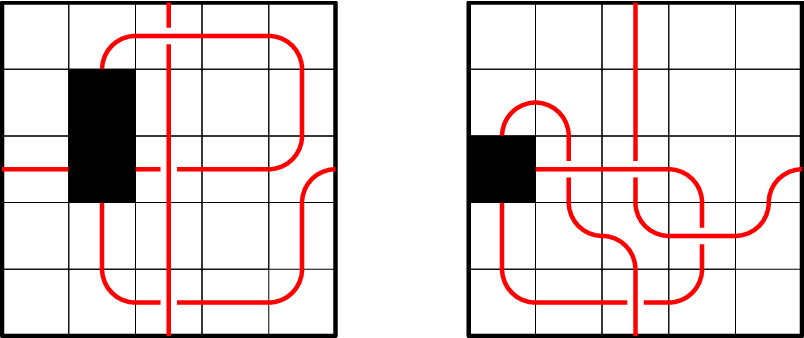}
\caption {Non-examples of tangle mosaics. These are forbidden because there are two adjacent $\blacksquare$ tiles (left) or a $\blacksquare$ tile is along the boundary (right).}
\end{figure}

\begin{remark}\label{tameRemark}
    Definition \ref{rv24def} implies that a $1\times 1$ comprised of a single $\blacksquare$ is not a valid $\rv{2,4}$ tangle mosaic. In other words, we require every $\rv{2,4}$ tangle mosaic to contain tiles in $\mathbb{T}$.
\end{remark}


If we consider the diagram of a given knot or tangle, we can use standard arguments to overlay a grid on the plane and associate a tile from $\mathbb{T}$ with each square to construct a mosaic. By considering an analogous argument for $\rv{2,4}$ knots and tangles, we can also associate squares containing rigid vertices with $\blacksquare$ tiles. Therefore, we record the following representability result.


\begin{proposition}[Representability of Tame $\rv{2,4}$ Graphs]\label{rv-rep}
    Every tame $\rv{2,4}$ graph $G$ admits an $\rv{2,4}$ mosaic representative $M \in \mathbb{M}_{\emph{RV}}^{(n)}$ for some $n \in \mathbb{N}$. If $G$ is closed (no boundary), then $M \in \mathbb{K}_{\emph{RV}}^{(n)}$; if $G$ has $2m$ boundary points, then $M \in \mathbb{L}_{\emph{RV}}^{(n)}$ with $2m$ connection points along the boundary.
\end{proposition}

\begin{proof}
Choose a regular PL diagram of $G$ in the square so that each rigid vertex projects to a distinct point and each edge segment is tame. Refine the square grid so that every cell contains at most one local feature: a tame strand segment, an ordinary crossing, a turn, a rigid vertex, or no part of the diagram. Put the corresponding tile from $\mathbb{T}$ in each tame cell and put a $T_\infty$ tile in each cell containing a rigid vertex, with the incident strand directions recording its degree-$2$ or degree-$4$ connection data. A small planar isotopy and further refinement ensure that adjacent tiles match and that no $T_\infty$ tile lies on the boundary. Thus the resulting mosaic is an $\rv{2,4}$ mosaic representative of $G$. If $G$ is closed then no boundary connection points occur, so the representative lies in $\KK_{\text{RV}}^{(n)}$; if $G$ has $2m$ boundary points, these become the $2m$ boundary connection points of an element of $\LL_{\text{RV}}^{(n)}$.
\end{proof}

Just as there are special Reidemeister moves governing interactions of knot and tangle strands with rigid vertices, we propose analogous moves describing equivalence in the $\rv{2,4}$ tangle mosaic setting. See Figure \ref{rv-mosaic-reid}.

\begin{figure}[H]
\includegraphics[width=\linewidth]{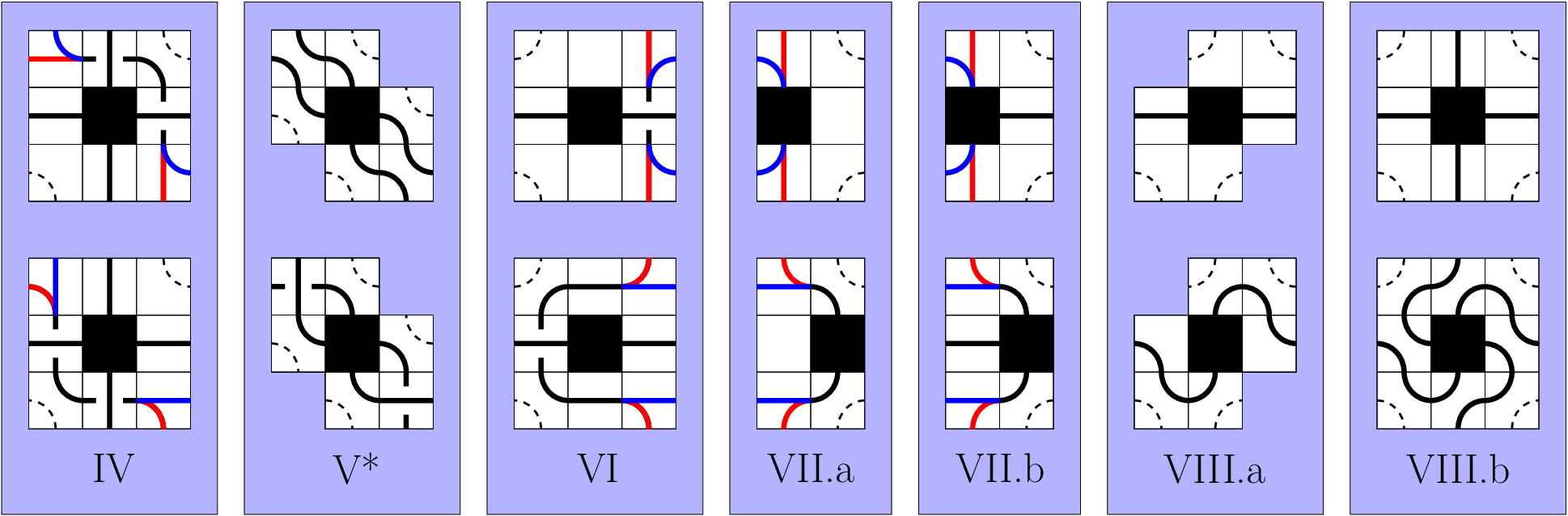}
\caption{The new $\rv{2,4}$ mosaic Reidemeister moves corresponding to IV, V*, and VI for $\rv{2,4}$ graphs, accounting for the $\blacksquare$ tile. We include the pictured moves as well as their mirror images and their rotations in our expanded set of mosaic equivalences.}\label{rv-mosaic-reid}
\end{figure}





\begin{lemma}[Boundary Adjustment]\label{bd_adjustment}
    Given any $1$- or $2$-tangle $\rv{2,4}$ mosaic, there is an equivalent odd-dimensional mosaic tangle such that:
    \begin{itemize}
        \item The boundaries of strands are in the center tile along the edges of the mosaic,
        \item For $1$-tangle $\rv{2,4}$, connection points of the mosaic occur on opposite sides.
    \end{itemize} 
\end{lemma}

\begin{proof}
    Use $\rv{2,4}$ tangle injections (see Remark \ref{rv-tangle-injections}) to map to a larger, equivalent, odd-dimensional mosaic. Then, planar isotop the connection points to be in the center tiles of each side. For a $1$-tangle $\rv{2,4}$ mosaic, ensure that the two connection points are isotoped to opposite sides.
\end{proof}

Note that these boundary-adjusted $1$- or $2$-tangle $\rv{2,4}$ mosaics thus have the same connection points assigned to $T_\infty = \blacksquare$.

\begin{figure}[H]
\includegraphics[width=0.7\linewidth]{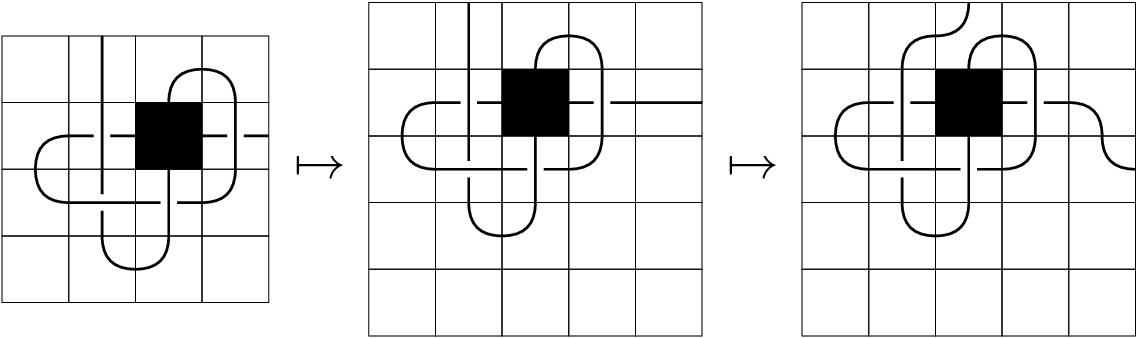}
\caption{Boundary adjustment of a $1$-tangle $\rv{2,4}$ mosaic.}\label{adjust}
\end{figure}

\begin{definition}[$\rv{2,4}$ Tangle Mosaic Equivalence]\label{tangequiv}
    We call two $\rv{2,4}$ mosaics $M$ and $M'$ {\bf boundary equivalent}, denoted $M\underset{b}{\sim}M'$, if one is a boundary adjustment of the other. Furthermore, two $\rv{2,4}$ mosaic tangles $M$ and $N$ are {\bf $\rv{2,4}$ tangle mosaic equivalent} if there exists $M\underset{b}{\sim}M'$ and $N\underset{b}{\sim}N'$ such that $M'$ and $N'$ are related by a finite sequence of mosaic $\rv{2,4}$ moves. 
\end{definition}

Note that, according to the notation in Definition \ref{tangequiv}, boundary-equivalent $\rv{2,4}$ mosaic tangles $M$ and $M'$ may (and typically do) have different mosaic dimensions. Similarly, equivalent mosaic tangles $M$ and $N$ may have different mosaic dimensions. However, $M'$ and $N'$ are assumed to have the same mosaic dimension. Given this new notion of equivalence, we generalize Conjecture \ref{conj:lk-tangles} to $\rv{2,4}$ tangle mosaics.

\begin{conjecture}[Topological Equivalence Conjecture for $\rv{2,4}$ tangles]\label{lk-conj-rvtangle}
Let $l_1$ and $l_2$ be two $\rv{2,4}$ tangles, and let $L_1$ and $L_2$ be two arbitrarily chosen $\rv{2,4}$ tangle $n$-mosaic representatives of $l_1$ and $l_2$, respectively, with the same connection points along their boundaries. Then $l_1$ and $l_2$ are ambient isotopic if and only if the representative mosaics $L_1$ and $L_2$ are $\rv{2,4}$ tangle mosaic equivalent.
\end{conjecture}

Note that Proposition \ref{rv-rep} establishes that every tame $\rv{2,4}$ graph admits an $\rv{2,4}$ mosaic representative, providing the ``representability'' direction implicit in the conjecture above. Additionally, although the forward implication remains open, we record an ambient-isotopy form of the backward implication for later use.


\begin{proposition}\label{prop:lk-easy-rv}
Let $L_1$ and $L_2$ be $\rv{2,4}$ tangle $n$-mosaics related by a finite sequence of $\rv{2,4}$ mosaic Reidemeister moves. Then there exists an ambient isotopy $\tilde{h} \colon [0,n]^2 \times [0,1] \to [0,n]^2 \times [0,1]$ taking the tame-tile arcs of $L_1$ to those of $L_2$, and the underlying $\rv{2,4}$ tangles are ambient isotopic.
\end{proposition}

\begin{proof}
This is the backward implication of Conjecture~\ref{lk-conj-rvtangle}: each elementary $\rv{2,4}$ mosaic Reidemeister move is supported on a union of tile sub-cubes in $[0,n]^2 \times [0,1]$ and realizes a PL ambient isotopy of that sub-region rel boundary. Concatenating finitely many such isotopies yields~$\tilde{h}$.
\end{proof}

\section{Tree Mosaic Knots}

\subsection{Mosaic Embeddings}

Let us begin with some important technical details that will enable us to define our main object of study, the tree mosaic, a discretized representation of a class of wild knots. We first define a way to take an $n$-dimensional mosaic and replace it with a topologically equivalent $pn$-dimensional mosaic for some natural number $n$ and odd natural number $p$. This will enable us to rigorously describe a process whereby we insert a tangle $p$-mosaic into a $\blacksquare$ tile of an $\rv{2,4}$ knot or tangle $n$-mosaic.









For our first definition of this section, we encourage the reader to refer to the example in Figure \ref{pzoom} to get an intuitive idea of what the definition is formalizing.

\begin{definition}
    For $p = 2q -1$ odd, the \textbf{$p$-zoom} is a map
    $$z_p: \mathbb{M}^{(n)}_{\text{RV}} \longrightarrow \mathbb{M}_{\text{RV}}^{(pn)}$$
    
    where the $ij$-th tile of a mosaic $M^{(n)}$ determines a submosaic in $N^{(pn)}$ for  $p(i-1)< k \leq pi$ and $p(j-1)< \ell \leq pj$. Let $D(M^{(n)}_{ij}) \subseteq \{N,S,E,W\}$ denote the sides of $M^{(n)}_{ij}$ with connection points, using the specified connection data when $M^{(n)}_{ij}=\blacksquare$. Then
    
    \begin{equation*}
    N^{(pn)}_{k \ell} = 
        \left\{
        \begin{array}{lr}
            M^{(n)}_{ij}, & \text{if } \ k = p(i-1)+q \\
            & \text{ and } \ \ell = p(j-1)+q\\
            \\
            \five, & \text{if } \ k = p(i-1)+q,\ \ell > p(j-1)+q,\ \text{and } E \in D(M^{(n)}_{ij})\\
            \\
            \five, & \text{if } \ k = p(i-1)+q,\ \ell < p(j-1)+q,\ \text{and } W \in D(M^{(n)}_{ij})\\
            \\
            \six, & \text{if } \ \ell = p(j-1)+q,\ k > p(i-1)+q,\ \text{and } S \in D(M^{(n)}_{ij})\\
            \\
            \six, & \text{if } \ \ell = p(j-1)+q,\ k < p(i-1)+q,\ \text{and } N \in D(M^{(n)}_{ij})\\
            \\
            \zero , & \text{otherwise.}
        \end{array}
        \right.
    \end{equation*}



\end{definition}


\begin{figure}[H]
\includegraphics[width=.8\linewidth]{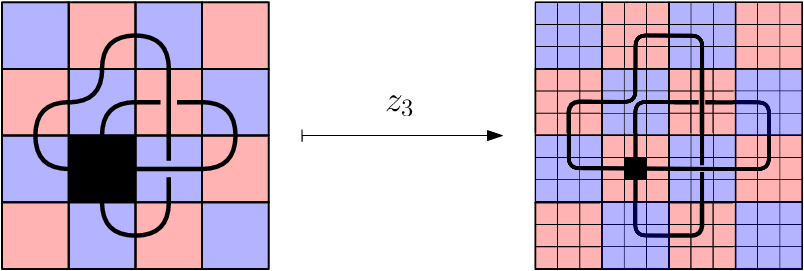}
\caption{The $3$-zoom map applied to an $\rv{2,4}$ mosaic.}\label{pzoom}
\end{figure}

For this next definition, we show how a tangle mosaic $M$ can be inserted into another (knot or tangle) mosaic  $N$. As the reader considers this definition, we recommend looking at the example in Figure \ref{embedfigure}.

\begin{definition}\label{embed}
    Let $m$ be odd and $(M,N)\in \LL_{\text{RV}}^{(m)} \times \MM_{\text{RV}}^{(n)}$ and suppose that the tile in the $(i, j)$ position of $N$ is a  $\blacksquare$ tile. Then the \textbf{fundamental embedding of mosaics} $x_{ij}$ is the map
    $$x_{ij}: (M,N)\in \LL_{\text{RV}}^{(m)} \times \MM_{\text{RV}}^{(n)} \longrightarrow \MM_{\text{RV}}^{(mn)}$$
   where $N' = z_m(N)$ is the $m$-zoom of $N$ and, for $1 \leq t,u \leq m$, a pair $(M,N)$ is sent to mosaic $W^{(mn)}$ given by
    $$W^{(mn)}_{k\ell} = \begin{cases}
        M_{tu} &\text{if  $(k,\ell) = (m(i-1)+t,m(j-1)+u)$} \\
        N'_{k\ell} &\text{otherwise}.
    \end{cases}$$

    Note that we can similarly define an embedding of mosaic tangles into mosaic {\em tangles} (rather than into mosaic knots), i.e. $(M,N)\in \LL_{\text{RV}}^{(m)} \times \LL_{\text{RV}}^{(n)}$.

    In either case, we refer to $W^{(mn)}$ as the \textbf{insertion} of $M$ into $N$.

    More generally, an embedding of mosaics can take many forms, and may not be an explicit embedding into a particular $(i,j)$th entry (e.g.\ the embedded mosaic may be flipped, or padded with a boundary adjustment, etc.); the formal framework is given in Definitions~\ref{def:D4-action} and~\ref{def:oriented-embed} below. The location and orientation of the embedding will be clear from context. We denote the space of all mosaic embeddings $\textsf{Hom}_\text{RV}$.
\end{definition}

\begin{figure}[H]
\includegraphics[width=.7\linewidth]{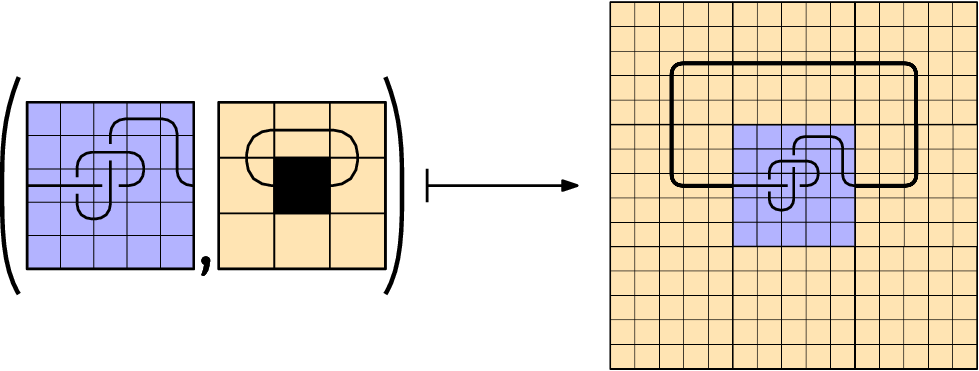}
\caption{An example of the mosaic embedding $x_{22}: \LL_{\text{RV}}^{(5)} \times \MM_{\text{RV}}^{(3)}\longrightarrow \MM_{\text{RV}}^{(15)}$.}\label{embedfigure}
\end{figure}

Notice that $\rv{2,4}$ mosaic moves $V^*$, $VIII.a$, and $VIII.b$ from Figure \ref{rv-mosaic-reid} flip and rotate a $\blacksquare$ rigid vertex. To later account for these moves at the tree mosaic level, we formalize the symmetries of the square as a group action on the space of $\rv{2,4}$ mosaics.

\begin{definition}[Dihedral Group $D_4$]\label{def:D4}
    The \textbf{dihedral group of the square} is the group
    \[
        D_4 = \langle\, r,\, f \mid r^4 = f^2 = e,\; frf^{-1} = r^{-1}\,\rangle,
    \]
    where $r$ denotes counterclockwise rotation by $90^\circ$ and $f$ denotes reflection across the NW--SE diagonal.
\end{definition}

\begin{definition}[$D_4$-Action on $\MM_\text{RV}$]\label{def:D4-action}
    We define a left action of $D_4$ on $\MM_{\text{RV}}^{(m)}$ by specifying the action of the generators $r$ and $f$ on an $m$-mosaic $M$. Set $k = m + 1 - j$. Then:
\begin{multicols}{2}
    $$(f \cdot M)_{ij} = \begin{cases}
        \one &\text{if }M_{ji} =  \three\\
        \three &\text{if }M_{ji} =  \one\\
        \five &\text{if }M_{ji} = \six \\
        \six &\text{if }M_{ji} =  \five \\
        M_{ji} &\text{otherwise}\\
    \end{cases}$$
\columnbreak
    $$(r \cdot M)_{ij} = \begin{cases}
        \one &\text{if }M_{ki} =  \two\\
        \two &\text{if }M_{ki} =  \three\\
        \three &\text{if }M_{ki} =  \four\\
        \four &\text{if }M_{ki} =  \one\\
        \five &\text{if }M_{ki} =  \six\\
        \six &\text{if }M_{ki} =  \five\\
        \seven &\text{if }M_{ki} =  \eight\\
        \eight &\text{if }M_{ki} =  \seven\\
        \nine &\text{if }M_{ki} =  \ten\\
        \ten &\text{if }M_{ki} =  \nine\\
        M_{ki} &\text{if }M_{ki} = \blacksquare,\zero \\
    \end{cases}$$
\end{multicols}
    That is, $f$ acts by NW--SE diagonal reflection (transposing indices and swapping tiles that exchange N/S and E/W connections) and $r$ acts by counterclockwise $90^\circ$ rotation. Since $r$ and $f$ satisfy the relations of $D_4$, this extends to a well-defined action of $D_4$ on $\MM_{\text{RV}}^{(m)}$. Moreover, since each generator preserves suitably-connectedness, the action restricts to the subspaces $\KK_{\text{RV}}^{(m)}$ and $\LL_{\text{RV}}^{(m)}$.
\end{definition}

\begin{definition}[Oriented Embeddings]\label{def:oriented-embed}
    For $\sigma \in D_4$ and a mosaic embedding $x_{ij}$ as in Definition \ref{embed}, the \textbf{$\sigma$-oriented embedding} is the map
    \[
        x_{ij}^\sigma(M, N) = x_{ij}(\sigma \cdot M,\, N).
    \]
    The fundamental embedding $x_{ij}$ of Definition~\ref{embed} is the special case $\sigma = e$, the identity element of $D_4$. In particular, $x_{ij}^r$ and $x_{ij}^f$ are the \textbf{rotated embedding} and \textbf{flipped embedding}, respectively. More generally, $\textsf{Hom}_\text{RV}$ includes all $\sigma$-oriented embeddings $x_{ij}^\sigma$ for $\sigma \in D_4$.
\end{definition}

\begin{figure}[t]
    \centering
    \captionsetup[subfigure]{labelformat=empty}
    \begin{subfigure}[b]{0.2\textwidth}
        \centering
        \includegraphics[width=\textwidth]{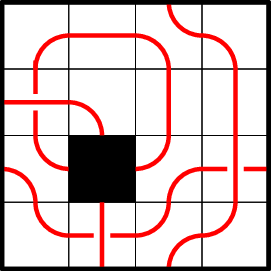}
        \caption{$M$}
    \end{subfigure}
    \hfill
    \begin{subfigure}[b]{0.2\textwidth}
        \centering
        \includegraphics[width=\textwidth]{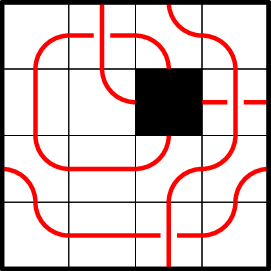}
        \caption{$f \cdot M$}
    \end{subfigure}
    \hfill
    \begin{subfigure}[b]{0.2\textwidth}
        \centering
        \includegraphics[width=\textwidth]{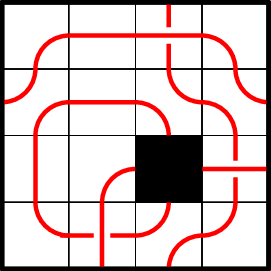}
        \caption{$r \cdot M$}
    \end{subfigure}
    \hfill
    \begin{subfigure}[b]{0.2\textwidth}
        \centering
        \includegraphics[width=\textwidth]{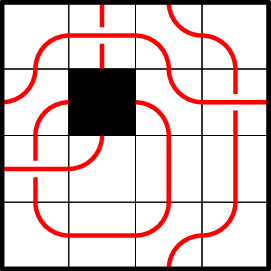}
        \caption{$r \cdot (f \cdot M)$}
    \end{subfigure}
    \caption{The $D_4$-action on an $\rv{2,4}$ mosaic $M$: identity, reflection $f$, rotation $r$, and their composition $rf$.}
\end{figure}

We now proceed to assign $\rv{2,4}$ mosaics to the vertices of a rooted tree with labeling on the edges dictating the procedure of insertion. In particular, rather than limiting ourselves to just finite trees (in which case we retain the theory of expressing tame knots, cf.\ Proposition \ref{geometric-realization-prop}), we allow infinitely many vertices in our tree. 

\subsection{Definition of a Tree Mosaic}

A \textbf{graph} $G$ consists of a set $V(G)$ of \textbf{vertices} and a family $E(G)$ of unordered pairs of elements of $V(G)$ called \textbf{edges}. If $V(G)$ and $E(G)$ are both countably infinite, then $G$ is a \textbf{countable graph}. The \textbf{degree} of a vertex $v \in V(G)$ is the cardinality of the set of edges incident to $v$, denoted by $\deg(v)$, and may be finite or infinite. 

An infinite graph is \textbf{locally finite} if each of its vertices has finite degree. For a locally finite graph $G$ and vertex $v \in V(G)$, we use $N(v)\subseteq V(G)$ to denote the neighbors of the vertex $v$ in $G$. Finally, if $G$ is a tree, then we may distinguish a vertex $v \in V(G)$ as being a \textbf{root}, and there exists a unique path from every vertex of $G$ to $v$. Therefore if $G$ is \textbf{rooted} at a vertex $v$, there exists a unique orientation on all of its edges pointing ultimately to $v$, hence we may consider it as an oriented graph. For two neighboring vertices on an oriented tree, the vertex pointed to is called the {\bf parent} while the other is called the {\bf child}. Notice, a vertex may have at most one parent and, if the tree is locally finite, it may have finitely many children. In the mosaic data below, we write an edge as $(u,w)$ with the parent $u$ listed first and the child $w$ listed second.

The following are the main definitions of our paper. Definition \ref{treemosaic} gives a general description of tree mosaics, while Definition \ref{suitable_tree} is the refinement we will use to characterize wild knots.

\begin{definition} \label{treemosaic}
A \textbf{tree mosaic} $M$ is a labeled rooted tree $T$, where
\begin{itemize}
    \item $T=(V,E)$ is a countable locally-finite rooted tree,
    \item $F: V \to \mathbb{M}_\text{RV}$ associates an $\rv{2,4}$-mosaic to every vertex, and
    \item $i: E \to \textsf{Hom}_{\text{RV}}$ associates an oriented embedding of mosaics (cf. Def. \ref{def:oriented-embed})  to every edge.
\end{itemize}

We say that a tree mosaic is \textbf{well-defined} if, for every vertex $v$, there is an order-preserving bijection between the $\blacksquare$ tiles of $F(v)$ and the children of $v$, and if the embedding $i(v,w)$ targets the $\blacksquare$ tile corresponding to the child $w$ under this bijection.
\end{definition}


\begin{definition} \label{suitable_tree}
    A well-defined tree mosaic $W=(T, F, i)$ is said to be \textbf{suitably connected} if 
    \begin{itemize}
        \item the root of the tree is assigned an $\rv{2,4}$ knot mosaic (from $\KK_{\text{RV}}$), 
        \item all subsequent vertices are assigned $\rv{2,4}$ tangle mosaics (from $\LL_{\text{RV}}$) of odd dimension at least $3$, and
        \item for every edge $(v,w) \in E(T)$, if $i(v,w)$ targets the $T_\infty$ tile of $F(v)$ at position $(a,b)$, then the boundary connection points of $F(w)$ match the number, sides, and strand directions of the connection points of that $T_\infty$ tile.
    \end{itemize}
    
The well-defined criterion ensures that every $\blacksquare$ tile is associated with embedding one of its mosaic children, and prevents having too many such squares. The suitably-connected definition ensures that the $\rv{2,4}$ mosaics are themselves connected, and connect properly when they are embedded into each other. From now on, we assume that all tree mosaics are well-defined and suitably connected, unless otherwise stated.
\end{definition}

\begin{definition}\label{subtree}
    If $\tilde{F}$ and $\tilde{i}$ are the restrictions of $F$ and $i$ to some subtree $\tilde{T}\subset T$, then $\tilde{M}=(\tilde{T},\tilde{F},\tilde{i})$ is called a \textbf{subtree mosaic} of $M$.

    We may restate this as: all definitions of a tree mosaic apply, except that the root of the tree is assigned an $\rv{2,4}$ tangle mosaic (from $\LL_{\text{RV}}$), instead of an $\rv{2,4}$ knot mosaic (from $\KK_{\text{RV}}$).
\end{definition}


\begin{figure}[t]
    \centering
    \newcommand{\TreeMosaicPanelWidth}{0.12\textwidth}
    \newcommand{\TreeMosaicPanelSep}{0.08\textwidth}
    \captionsetup[subfigure]{}

    \begin{subfigure}[b]{\TreeMosaicPanelWidth}
        \centering
        \includegraphics[angle=-90,width=0.85\textwidth]{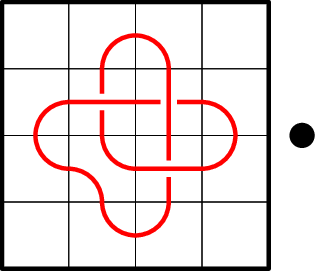}
        \caption{Trefoil mosaic on a single vertex}
    \end{subfigure}
    \hspace{\TreeMosaicPanelSep}
    \begin{subfigure}[b]{\TreeMosaicPanelWidth}
        \centering
        \includegraphics[angle=-90,width=\textwidth]{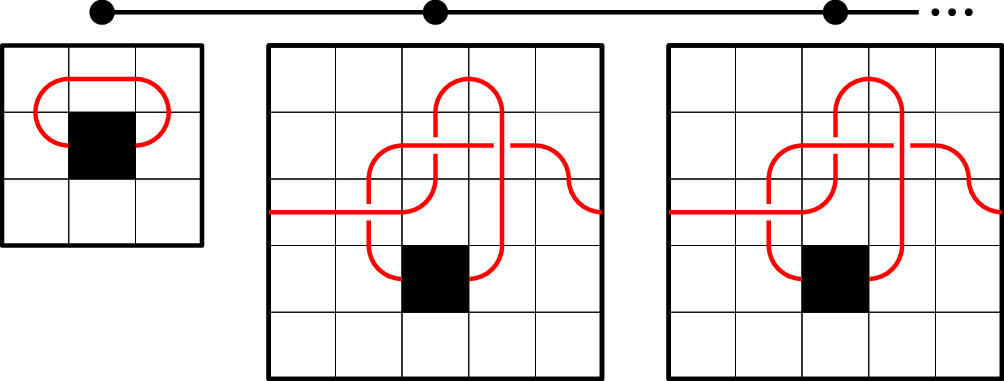}
        \caption{Connect sum of trefoils}
    \end{subfigure}
    \hspace{\TreeMosaicPanelSep}
    \begin{subfigure}[b]{\TreeMosaicPanelWidth}
        \centering
        \includegraphics[angle=-90,width=\textwidth]{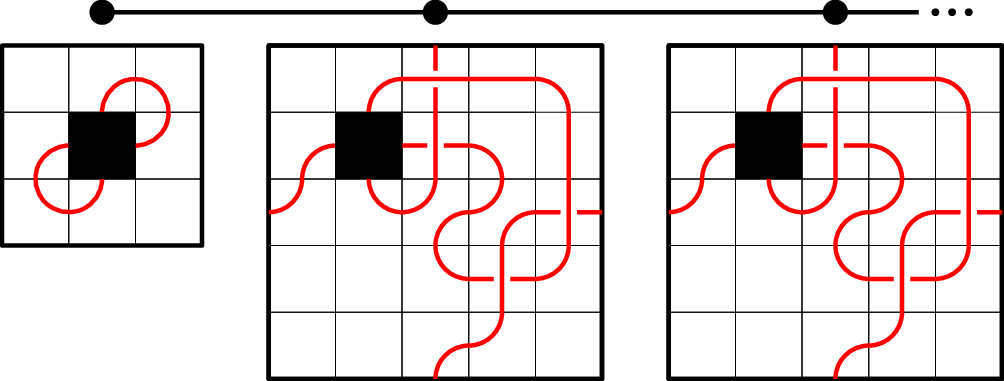}
        \caption{Infinite slipknot from Figure \ref{slip}}
    \end{subfigure}
    \hspace{\TreeMosaicPanelSep}
    \begin{subfigure}[b]{\TreeMosaicPanelWidth}
        \centering
        \includegraphics[angle=-90,width=\textwidth]{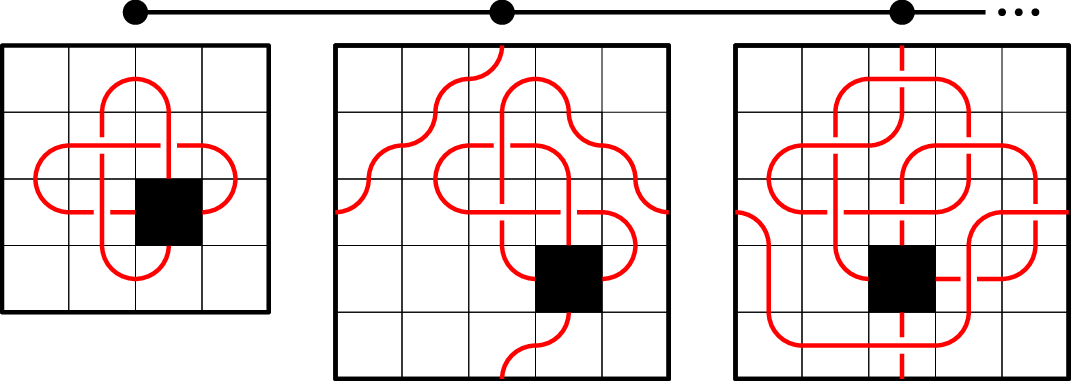}
        \caption{Non-repeating tree mosaic}
    \end{subfigure}

    \caption{Examples of tree mosaics. ({\tiny A}) captures the usual notion of a knot mosaic. ({\tiny B}) and ({\tiny C}) demonstrate our ability to describe classical wild knot constructions. ({\tiny D}) shows that the tangles associated with each vertex need not be constant.}
\end{figure}



\begin{remark}
    Note that above we may actually replace $\MM_{\text{RV}}$ with subsets of $\mathbb{R}^3$ (or, with more care, various other topological spaces) with some boundary conditions for connectedness later. Then this framework can be used to study recursive topological spaces that exhibit some degree of self-similarity.
\end{remark}

\begin{definition}[Tile realization] \label{tile-realization}
For each tile $T_k \in \mathbb{T} = \{T_0, \ldots, T_{10}\}$, define the \textbf{arc domain}:
$$A(T_k) = \begin{cases}
    \varnothing & k = 0 \\
    [0,1] & k \in \{1, \ldots, 6\} \\
    [0,1] \sqcup [0,1] & k \in \{7, \ldots, 10\}
\end{cases}$$
Each tile has a canonical PL embedding $\varphi_k: A(T_k) \to [0,1]^3$ realizing its arcs within the unit cube. The four possible \textbf{connection points} are the face centers: $c_N = (\frac{1}{2}, 1, \frac{1}{2})$, $c_S = (\frac{1}{2}, 0, \frac{1}{2})$, $c_E = (1, \frac{1}{2}, \frac{1}{2})$, $c_W = (0, \frac{1}{2}, \frac{1}{2})$. Each arc in $\varphi_k(A(T_k))$ is a PL path connecting two of these face centers according to the tile's combinatorial type, with the interior of the arc contained in $(0,1)^3$.

For the $T_\infty$ tile, the arc domain depends on the degree of the rigid vertex: $A(T_\infty) = [0,1]$ for degree $2$ (a transverse pair of lines connecting opposite faces), and $A(T_\infty) = [0,1] \sqcup [0,1]$ for degree $4$ (two lines, one connecting $c_N$ to $c_S$ and one connecting $c_E$ to $c_W$).
\end{definition}

\begin{figure}[H] \label{tile-representation}
\includegraphics[width=.5\linewidth]{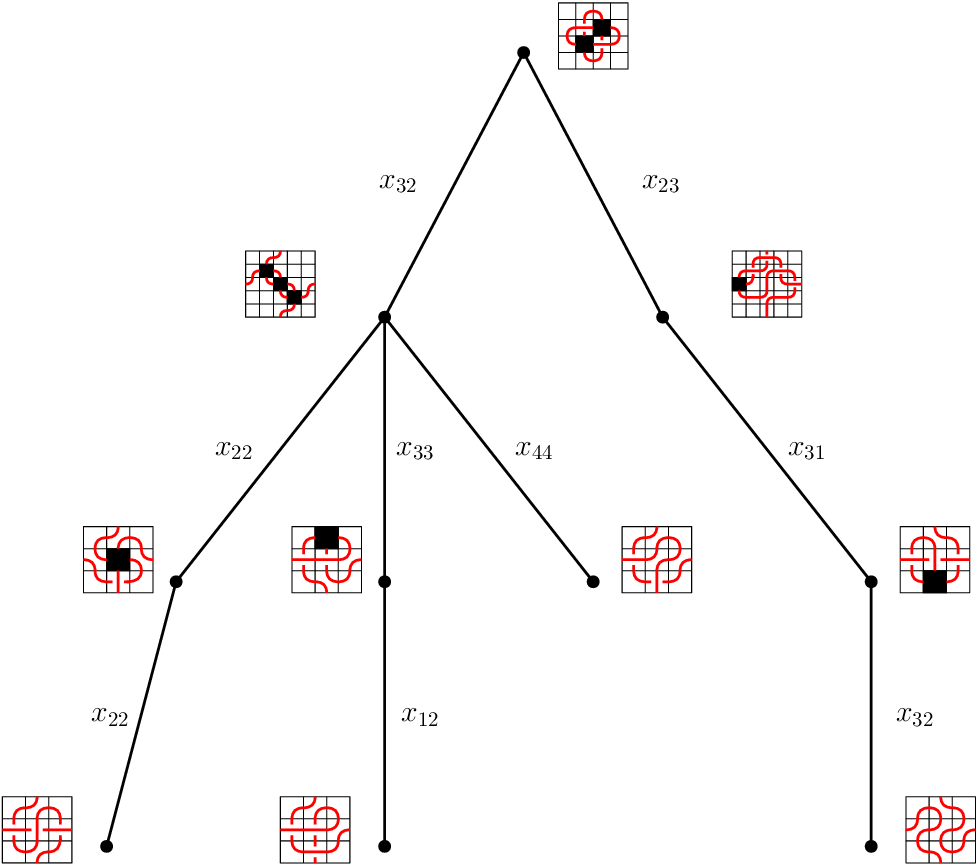} \qquad \qquad
\includegraphics[width=.35\linewidth]{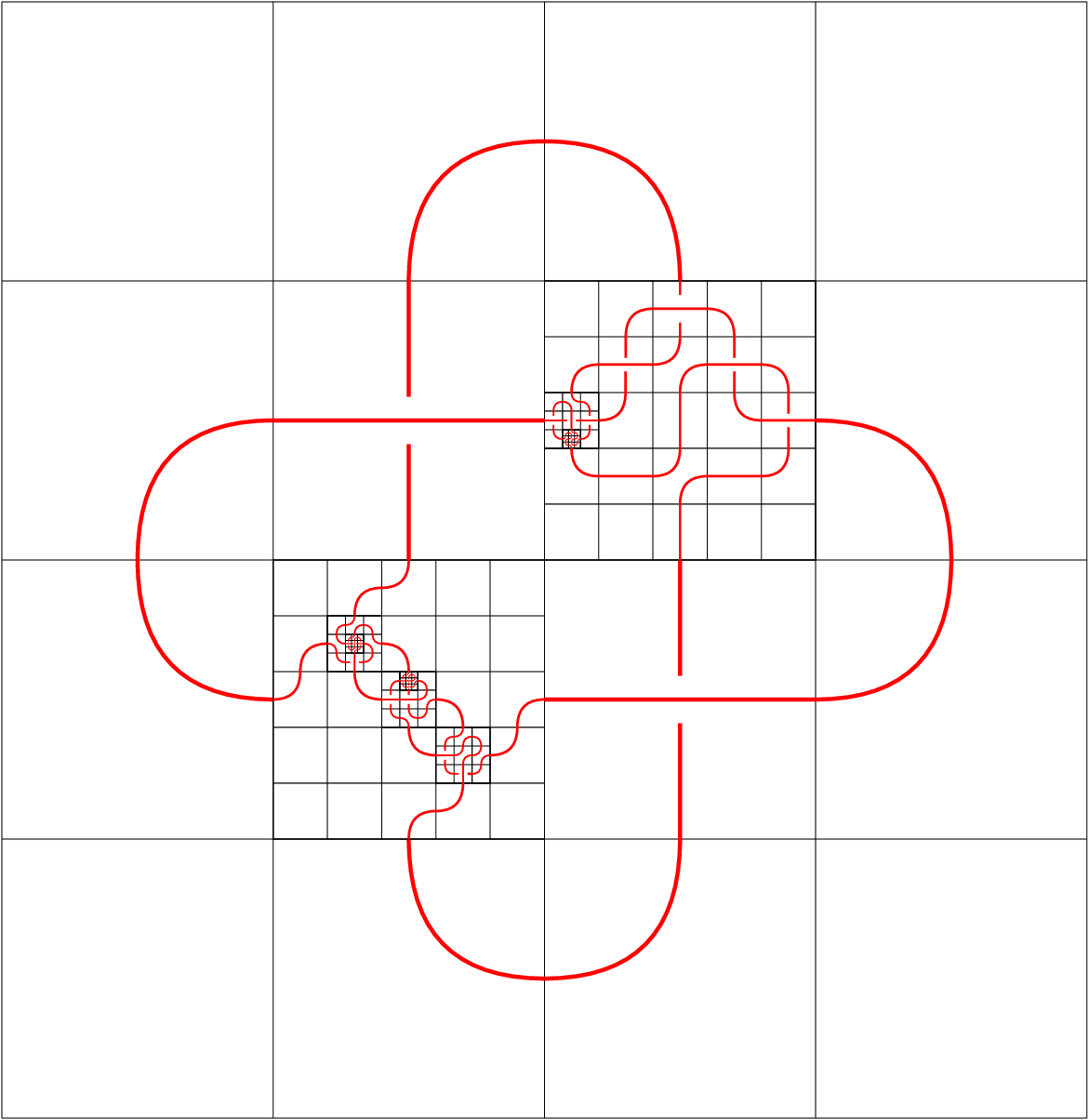} 
\caption{We can create a visual for the tree mosaic structure by drawing the mosaic embeddings directly. The right-hand picture may be interpreted as a ``top-down view" of the tree.}
\end{figure}

\begin{definition}[Geometric realization of a mosaic] \label{mosaic-realization}
Let $M \in \mathbb{M}^{(n)}$ be a suitably connected $n$-mosaic. The \textbf{geometric realization} $|M| \subset \mathbb{R}^3$ is the embedded $1$-complex obtained by placing tile $M_{ij}$ in the region $[i-1, i] \times [j-1, j] \times [0,1]$:
$$|M| = \bigcup_{i,j=1}^{n} \tau_{(i-1, j-1, 0)}\big(\varphi_{M_{ij}}(A(M_{ij}))\big)$$
where $\tau_{\mathbf{v}}: \mathbb{R}^3 \to \mathbb{R}^3$ denotes translation by $\mathbf{v}$. Since $M$ is suitably connected, adjacent tiles share connection points, ensuring arcs join continuously across tile boundaries. Hence $|M| \subset [0,n]^2 \times [0,1]$ is a finite union of PL arcs.

Similarly, for $\rv{2,4}$ mosaics $M \in \mathbb{M}_{\text{RV}}^{(n)}$, each $T_\infty$ tile contributes its line(s) meeting at a \textbf{singular transverse crossing} at the cube center $(\frac{1}{2}, \frac{1}{2}, \frac{1}{2})$: for degree $4$, two lines cross transversally at this point; for degree $2$, a single line passes through, marked as a rigid vertex.
\end{definition}

\begin{figure}[h]
    \centering
    \captionsetup[subfigure]{labelformat=empty}
    \begin{subfigure}[b]{0.6\textwidth}
         \includegraphics[width=\textwidth]{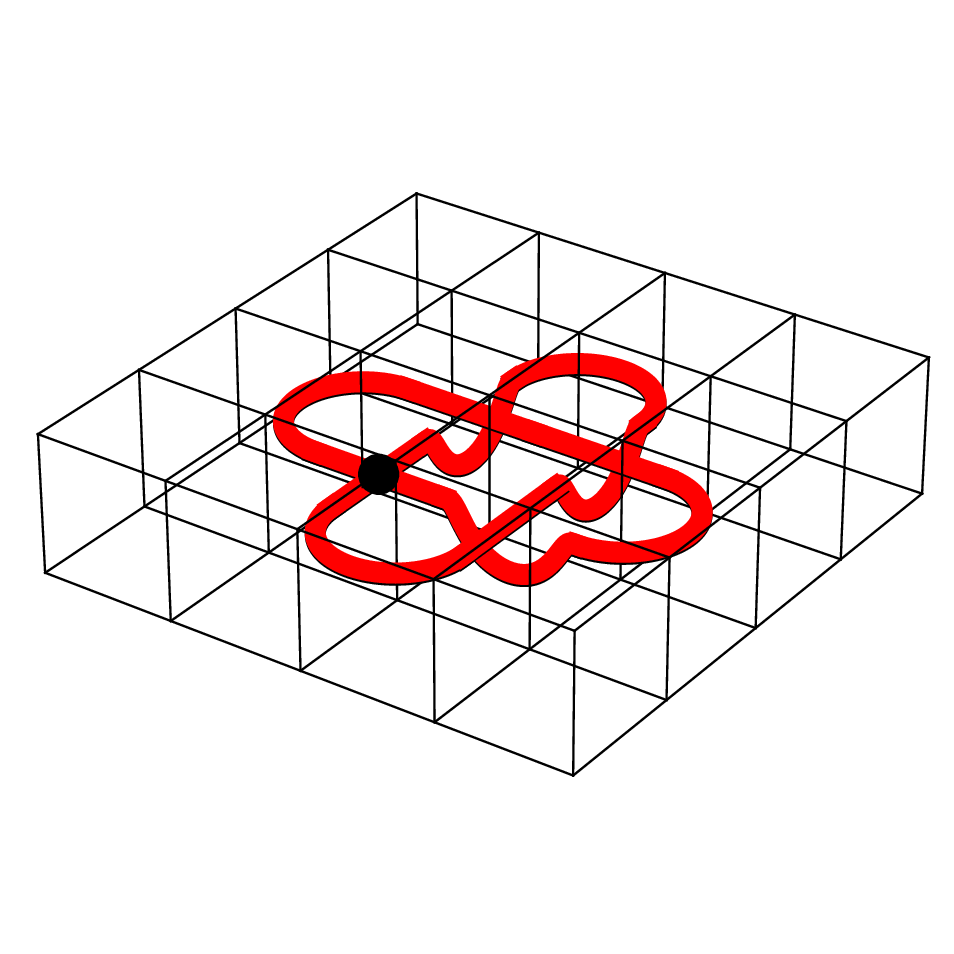}
        \vspace{-0.85in}
    \end{subfigure}
    \hfill
    \begin{subfigure}[b]{0.35\textwidth}
        \includegraphics[width=0.9\textwidth]{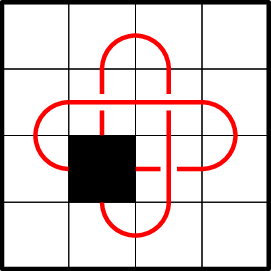}
        \vspace{0.4in}
    \end{subfigure}

    \caption{The geometric realization of a mosaic can be visualized as extending the mosaic tiles into cubes}
\end{figure}

\begin{definition}[Geometric realization of a tree mosaic] \label{tree-mosaic-realization}
Let $\mathcal{M} = (T, F, i)$ be a well-defined, suitably connected tree mosaic with root $r$. For each vertex $v \in V(T)$, let $\gamma_v = (r = v_0, v_1, \ldots, v_k = v)$ denote the unique path from the root to $v$, and let $d(v) = k$ denote the depth of $v$. Define:
\begin{enumerate}
    \item \textbf{Scale factor:} $\sigma: V(T) \to \mathbb{Q}_{>0}$ by $\sigma(r) = 1$, and if $v$ is a child of $u$ via edge $e$, then $\sigma(v) = \sigma(u) / \dim F(v)$. Explicitly, $\sigma(v) = \prod_{j=1}^{k} \frac{1}{\dim F(v_j)}$.
    \item \textbf{Position:} $\rho: V(T) \to \mathbb{Q}^3$ by $\rho(r) = \mathbf{0}$, and if the embedding $i(u,v)$ places $F(v)$ at position $(a,b)$ in $F(u)$ (i.e., into the $T_\infty$ tile at row $a$, column $b$), then $\rho(v) = \rho(u) + \sigma(u) \cdot (a-1, b-1, 0)$.
\end{enumerate}
For a subset $S \subset \mathbb{R}^3$, scalar $\lambda > 0$, and vector $\mathbf{w} \in \mathbb{R}^3$, write $\lambda \cdot S + \mathbf{w} := \{\lambda \mathbf{x} + \mathbf{w} : \mathbf{x} \in S\}$. The \textbf{geometric realization} is:
$$|\mathcal{M}| = \overline{\bigcup_{v \in V(T)} \Big( \sigma(v) \cdot |F(v)|_{\text{tame}} + \rho(v) \Big) \;\cup\; \bigcup_{(u,v) \in E(T)} \mathcal{C}(u,v)}$$
where $|F(v)|_{\text{tame}} \subset [0, \dim F(v)]^2 \times [0,1]$ denotes the realization of $F(v)$ using only tiles $T_0, \ldots, T_{10}$ (treating each $T_\infty$ tile as $T_0$, i.e., contributing $\varnothing$), and $\mathcal{C}(u,v)$ consists of \textbf{connecting segments}: for each connection point $q$ on the face of the $T_\infty$ sub-cube in $F(u)$ that corresponds to a connection point $q'$ on the boundary of $F(v)$ (as determined by the suitably connected condition), $\mathcal{C}(u,v)$ contains the straight-line segment $[\sigma(u) \cdot q + \rho(u),\; \sigma(v) \cdot q' + \rho(v)]$. Since $\sigma(v) \leq \sigma(u)/3$, the length of each connecting segment is at most $\sqrt{3}\,\sigma(u)$, tending to $0$ as depth increases. The closure is taken in $\mathbb{R}^3$.
\end{definition}

\begin{proposition} \label{geometric-realization-prop}
Let $\mathcal{M} = (T, F, i)$ be a well-defined, suitably connected tree mosaic.
\begin{enumerate}
    \item The geometric realization $|\mathcal{M}|$ is well-defined and compact.
    \item If the tree $T$ is finite, then $|\mathcal{M}|$ is a tame $\rv{2,4}$ graph.
    \item Each infinite path $\gamma = (v_0, v_1, v_2, \ldots)$ in $T$ determines a well-defined limit point $p_\gamma \in |\mathcal{M}|$.
\end{enumerate}
\end{proposition}

\begin{proof}
We prove each statement individually.

\begin{enumerate} 
	\item We first show $|\mathcal{M}|$ is bounded. Since $F(r) \in \mathbb{K}_{\text{RV}}$ is a knot mosaic, $|F(r)|_{\text{tame}} \subset [0, n_0]^2 \times [0,1]$ where $n_0 = \dim F(r)$. For any child $v$ of the root, the scaled and translated contribution $\sigma(v) \cdot |F(v)|_{\text{tame}} + \rho(v)$ lies entirely within the unit cube at position $(a-1, b-1, 0)$ in the root mosaic (where the corresponding $T_\infty$ tile was located). By induction, each $\sigma(v) \cdot |F(v)|_{\text{tame}} + \rho(v)$ is contained in $[0, n_0]^2 \times [0,1]$. Thus $|\mathcal{M}|$, being the closure of a bounded set, is bounded.

To show $|\mathcal{M}|$ is closed (hence compact), observe that the closure operation is included in the definition. The only limit points not already in the union are limits of sequences approaching infinite paths, which we address in (3).

	\item If $T$ is finite, the union is finite and each $|F(v)|_{\text{tame}}$ is a finite union of PL arcs. The suitably connected condition ensures arcs match at boundaries between parent and child mosaics. Hence $|\mathcal{M}|$ is a finite PL 1-complex, i.e., a tame $\rv{2,4}$ graph.

	\item Let $\gamma = (v_0, v_1, v_2, \ldots)$ be an infinite path from the root. We show $\{\rho(v_k)\}_{k=0}^{\infty}$ is a Cauchy sequence in $\mathbb{R}^3$.

First, we bound the distance between consecutive terms (positions between descending black boxes). For $k \geq 0$, let $(a_k, b_k)$ denote the row and column of the $T_\infty$ tile in $F(v_k)$ into which $F(v_{k+1})$ is embedded. The position update formula (cf.\ Definition \ref{tree-mosaic-realization}) states $\rho(v_{k+1}) = \rho(v_k) + \sigma(v_k) \cdot (a_k - 1, b_k - 1, 0)$, so:
\begin{align*}
\|\rho(v_{k+1}) - \rho(v_k)\|
&= \|\rho(v_k) + \sigma(v_k) \cdot (a_k - 1, b_k - 1, 0) - \rho(v_k)\| \\
&= \|\sigma(v_k) \cdot (a_k - 1, b_k - 1, 0)\| \\
&= \sigma(v_k) \cdot \|(a_k - 1, b_k - 1, 0)\| \qquad \text{(since $\sigma(v_k) > 0$)}.
\end{align*}
Since $F(v_k)$ is a $(\dim F(v_k))$-mosaic, the tile indices satisfy $1 \leq a_k, b_k \leq \dim F(v_k)$, hence $0 \leq a_k - 1, b_k - 1 \leq \dim F(v_k) - 1$. The norm is therefore bounded:
\begin{align*}
\|(a_k - 1, b_k - 1, 0)\|
&= \sqrt{(a_k - 1)^2 + (b_k - 1)^2 + 0^2} \\
&\leq \sqrt{(\dim F(v_k) - 1)^2 + (\dim F(v_k) - 1)^2} \\
&= (\dim F(v_k) - 1) \cdot \sqrt{2} \\
&< \sqrt{2} \cdot \dim F(v_k).
\end{align*}
Combining these, we obtain $\|\rho(v_{k+1}) - \rho(v_k)\| < \sqrt{2} \cdot \sigma(v_k) \cdot \dim F(v_k)$.

Now, since $\sigma(v_{k+1}) = \sigma(v_k) / \dim F(v_{k+1})$, we have $\sigma(v_k) \cdot \dim F(v_k) = \sigma(v_{k-1})$ for $k \geq 1$. Moreover, each non-root mosaic dimension satisfies $\dim F(v_j) \geq 3$ by Definition~\ref{suitable_tree}, so:
$$\sigma(v_k) = \prod_{j=1}^{k} \frac{1}{\dim F(v_j)} \leq \frac{1}{3^k}.$$

For $m > n \geq 1$, the triangle inequality gives:
$$\|\rho(v_m) - \rho(v_n)\| \leq \sum_{k=n}^{m-1} \|\rho(v_{k+1}) - \rho(v_k)\| < \sqrt{2} \sum_{k=n}^{m-1} \sigma(v_{k-1}) \leq \sqrt{2} \sum_{k=n}^{m-1} \frac{1}{3^{k-1}} = \frac{\sqrt{2}}{3^{n-1}} \sum_{j=0}^{m-n-1} \frac{1}{3^j}.$$
The geometric series is bounded by $\sum_{j=0}^{\infty} 3^{-j} = \frac{3}{2}$, so:
$$\|\rho(v_m) - \rho(v_n)\| < \frac{3\sqrt{2}}{2 \cdot 3^{n-1}} = \frac{3^2 \sqrt{2}}{2 \cdot 3^n} = \frac{9\sqrt{2}}{2} \cdot 3^{-n}.$$

Given $\varepsilon > 0$, choose $N \geq 1$ such that
\[
    \frac{9\sqrt{2}}{2} \cdot 3^{-N} < \varepsilon.
\]
Then for all $m > n \geq N$, we have
\[
    \|\rho(v_m) - \rho(v_n)\| < \varepsilon.
\]
Thus $\{\rho(v_k)\}$ is Cauchy.

Since $\mathbb{R}^3$ is complete, the limit $p_\gamma := \lim_{k \to \infty} \rho(v_k)$ exists and is unique for the fixed path $\gamma$. The point $p_\gamma$ lies in $|\mathcal{M}|$: choose any point of the nonempty tame realization inside each scaled descendant box
\[
    \rho(v_k) + [0,\; \sigma(v_k)\dim F(v_k)]^2 \times [0,\; \sigma(v_k)].
\]
The diameters of these boxes tend to zero, and $\rho(v_k) \to p_\gamma$, so these chosen points also converge to $p_\gamma$; the closure in the definition of $|\mathcal{M}|$ therefore contains $p_\gamma$.

\end{enumerate}

\end{proof}

\begin{definition}
    A tree mosaic $M=(T, F, i)$ is called \textbf{ray-repeating} if $T$ is an infinite ray such that for any two non-root vertices $w,v\in V(T)$ the associated $\rv{2,4}$ tangle mosaics are equal $F(w)=F(v).$
\end{definition}

Ray-repeating tree mosaics are a useful class of tree mosaics to draw examples from.

\begin{definition}
Let $M$ be a tree mosaic on $T$. We call an infinite path originating from the root a \textbf{combinatorial-wild point}.
\end{definition}


As we show below, every wild point (having $p$-index $\leq 4$) must be represented by a combinatorial wild point. Note the converse is not true, which is discussed in Theorem \ref{inf-ray-contract}.

\begin{lemma} \label{finite-wild-points}
A knot with isolated wild points has only finitely many wild points.
\end{lemma}

\begin{proof}
Local tameness is open along the knot: if $p \in K$ is locally tame, then there is an ambient ball $B$ about $p$ for which the pair $(B, B \cap K)$ is PL-equivalent to a standard unknotted arc in a ball. Any sufficiently nearby point of $K \cap B$ has a smaller ambient ball with the same property. Hence the locally tame points form an open subset of $K$, so the wild points $W \subset K$ form a closed subset. Since $W$ is also discrete (the wild points are isolated) and $K \cong \mathbb{S}^1$ is compact, $W$ is finite.
\end{proof}

\begin{theorem} \label{wild-knot-representability}
    If $K$ is a wild knot with isolated wild points, each of $p$-index at most $4$, then $K$ admits a tree mosaic representative. 
\end{theorem}

\begin{proof}
Let $K \subset \mathbb{R}^3$ be a wild knot with isolated wild points, each having $p$-index at most $4$. We construct a tree mosaic $M = (T, F, i)$ whose geometric realization is $K$.

\medskip
\noindent\textbf{Step 1: Finiteness of wild points.}
Let $W \subset K$ denote the set of wild points. By Lemma \ref{finite-wild-points}, $W$ is finite; write $W = \{p_1, \ldots, p_N\}$ for some $N \geq 1$. For each $p_j$, let $P_j := P(K, p_j) \in \{2, 4\}$ denote its $p$-index.

\medskip
\noindent\textbf{Step 2: Separating neighborhoods.}
Since $W$ is finite and the $p_j$ are distinct, we may choose $\varepsilon > 0$ such that:
\begin{enumerate}
    \item[(i)] $\varepsilon < \frac{1}{2} \min_{j \neq k} \|p_j - p_k\|$, so the balls $B_\varepsilon(p_j)$ are pairwise disjoint, and
    \item[(ii)] each $B_\varepsilon(p_j)$ contains no wild points other than $p_j$.
\end{enumerate}

\medskip
\noindent\textbf{Step 3: Constructing the root mosaic.}
For each $j \in \{1, \ldots, N\}$, by definition of $p$-index, there exists $\varepsilon_{j,1}$ with $0 < \varepsilon_{j,1} < \varepsilon$ such that $\partial B_{\varepsilon_{j,1}}(p_j)$ intersects $K$ in exactly $P_j$ points. Consider the compact subset
$$K_0 := K \setminus \bigcup_{j=1}^{N} \mathrm{int}(B_{\varepsilon_{j,1}}(p_j)).$$
Since $K_0$ contains no wild points, it is a tame arc (or union of tame arcs) with $P_j$ boundary points on each sphere $\partial B_{\varepsilon_{j,1}}(p_j)$. We view $K_0$ as an $\rv{2,4}$ graph by collapsing each sphere $\partial B_{\varepsilon_{j,1}}(p_j)$ to a single rigid vertex of degree $P_j$ (either $2$ or $4$), preserving the cyclic order of the incident strands.

By Proposition \ref{rv-rep}, there exists an $\rv{2,4}$ knot mosaic $M_0 \in \mathbb{K}_{\text{RV}}$ representing $K_0$. This mosaic $M_0$ contains exactly $N$ tiles of type $T_\infty$, one for each wild point neighborhood. Let $M_0$ be the mosaic assigned to the root $r$ of our tree.

\medskip
\noindent\textbf{Step 4: Constructing rays for each wild point.}
Fix $j \in \{1, \ldots, N\}$. We construct an infinite sequence of $\rv{2,4}$ tangle mosaics corresponding to nested annular regions around $p_j$.

Since $p_j$ has $p$-index $P_j$, there exist radii $\varepsilon_{j,1} > \varepsilon_{j,2} > \varepsilon_{j,3} > \cdots$ with $\varepsilon_{j,n} \to 0$ such that each sphere $\partial B_{\varepsilon_{j,n}}(p_j)$ meets $K$ in exactly $P_j$ points. For each $n \geq 1$, define the annular region
$$A_{j,n} := \left( \overline{B_{\varepsilon_{j,n}}(p_j)} \setminus \mathrm{int}(B_{\varepsilon_{j,n+1}}(p_j)) \right) \cap K.$$
Each $A_{j,n}$ is a tame tangle: it is compact, contains no wild points (since $p_j$ is the only wild point in $B_\varepsilon(p_j)$ and $p_j \in B_{\varepsilon_{j,n+1}}(p_j)$), and has exactly $P_j$ boundary points on each of its two boundary spheres.

Collapse the inner sphere $\partial B_{\varepsilon_{j,n+1}}(p_j)$ to a single rigid vertex of degree $P_j$. This turns $A_{j,n}$ into a tame $\rv{2,4}$ tangle with $P_j$ boundary points on the outer sphere and one $T_\infty$ vertex representing the next level of nesting. By Proposition \ref{rv-rep}, it admits an $\rv{2,4}$ tangle mosaic representative; after applying the boundary adjustment lemma, we may assume this representative has odd dimension. Denote the resulting mosaic by $M_{j,n} \in \mathbb{L}_{\text{RV}}$. It has $P_j$ connection points on its outer boundary and $P_j$ connection points incident to its $T_\infty$ tile.

\begin{figure}[h]
    \centering
    \captionsetup[subfigure]{labelformat=empty}
    \begin{subfigure}[b]{0.4\textwidth}
        \centering
        \includegraphics[width=\textwidth]{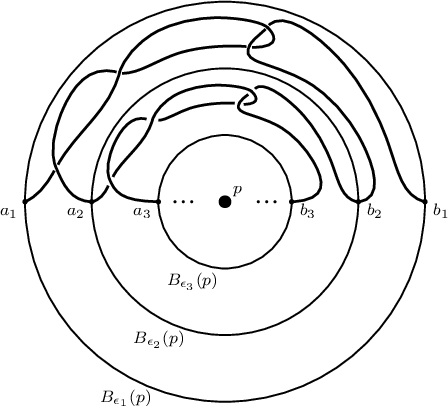}
        \caption{}
    \end{subfigure}
    \qquad
    \begin{subfigure}[b]{0.3\textwidth}
        \centering
        \includegraphics[width=\textwidth]{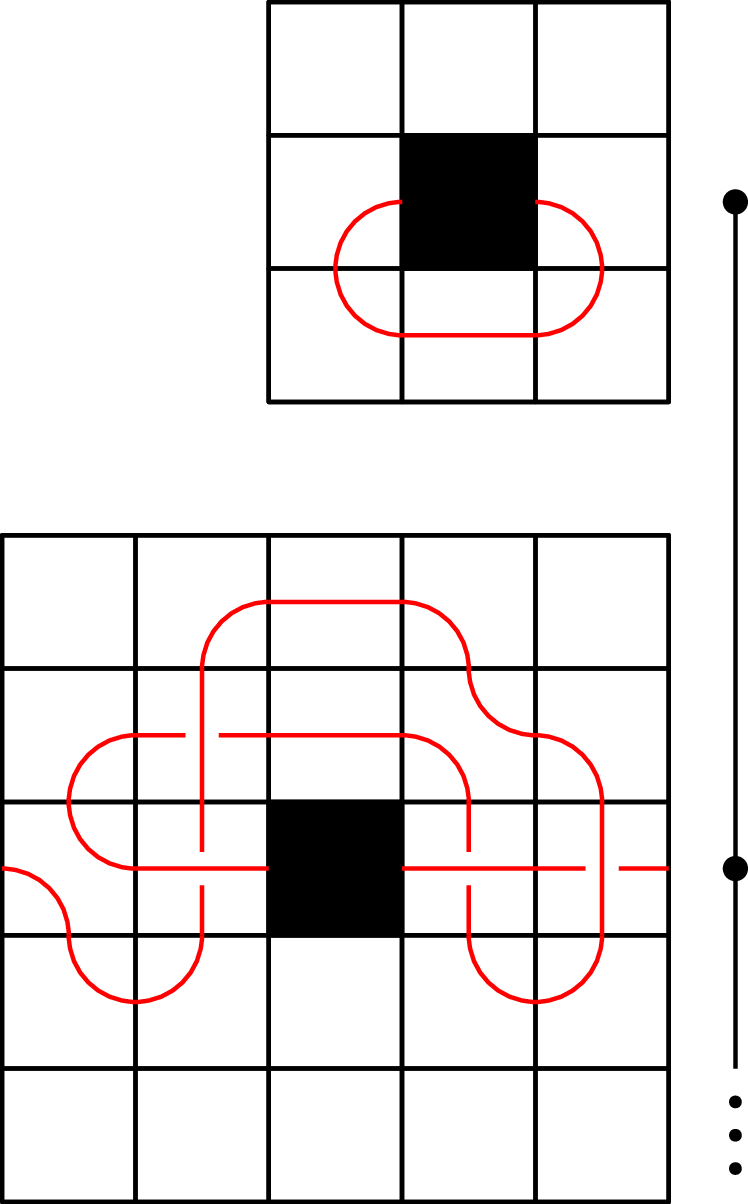}
        \caption{}
    \end{subfigure}
    \hfill
     \caption {Nested neighborhoods of an isolated wild point with $p$-index $2$. Each annular region $A_{j,n}$ is tame and admits a mosaic representative.}
\end{figure}

\medskip
\noindent\textbf{Step 5: Defining the tree structure.}
We define the rooted tree $T = (V, E)$ as follows:
\begin{itemize}
    \item The vertex set is $V = \{r\} \cup \{v_{j,n} : 1 \leq j \leq N, \, n \geq 1\}$.
    \item The root $r$ has $N$ children: $v_{1,1}, \ldots, v_{N,1}$.
    \item For each $j$ and $n \geq 1$, vertex $v_{j,n}$ has exactly one child: $v_{j,n+1}$.
\end{itemize}
Thus $T$ consists of the root $r$ with $N$ infinite rays attached, one ray $r \to v_{j,1} \to v_{j,2} \to \cdots$ for each wild point $p_j$.

\medskip
\noindent\textbf{Step 6: Assigning mosaics and embeddings.}
Define the mosaic assignment $F: V \to \mathbb{M}_{\text{RV}}$ by:
\begin{itemize}
    \item $F(r) = M_0$, the root mosaic constructed in Step 3.
    \item $F(v_{j,n}) = M_{j,n}$, the tangle mosaic for the $n$-th annular region around $p_j$.
\end{itemize}

For the embedding functions $i: E \to \textsf{Hom}_{\text{RV}}$, we specify:
\begin{itemize}
    \item For the edge $(r, v_{j,1})$: the embedding $i(r, v_{j,1})$ inserts $M_{j,1}$ into the $T_\infty$ tile of $M_0$ corresponding to wild point $p_j$, matching the $P_j$ connection points.
    \item For each edge $(v_{j,n}, v_{j,n+1})$: the embedding $i(v_{j,n}, v_{j,n+1})$ inserts $M_{j,n+1}$ into the unique $T_\infty$ tile of $M_{j,n}$, again matching connection points.
\end{itemize}
The connection points are matched using the boundary adjustment lemma to ensure consistent orientation.

\medskip
\noindent\textbf{Step 7: Verification.}
We verify that $M = (T, F, i)$ is a well-defined, suitably connected tree mosaic:
\begin{itemize}
    \item \emph{Well-defined}: Each $T_\infty$ tile has exactly one corresponding child vertex, and vice versa.
    \item \emph{Suitably connected}: The root mosaic $F(r) \in \mathbb{K}_{\text{RV}}$ is a knot mosaic, each $F(v_{j,n}) \in \mathbb{L}_{\text{RV}}$ is an odd-dimensional tangle mosaic, and connection points match across embeddings by construction.
\end{itemize}

\medskip
\noindent\textbf{Step 8: Geometric realization.}
It remains to show that the geometric realization $|M|$ (cf.\ Definition \ref{tree-mosaic-realization}) recovers $K$ up to ambient homeomorphism. The knot $K$ decomposes as
$$K = K_0 \cup \bigcup_{j=1}^{N} \left(\{p_j\} \cup \bigcup_{n=1}^{\infty} A_{j,n}\right),$$
where this follows from $\bigcap_{n=1}^{\infty} B_{\varepsilon_{j,n}}(p_j) = \{p_j\}$ (since $\varepsilon_{j,n} \to 0$). By construction, the root mosaic represents $K_0$ with one $T_\infty$ tile for each collapsed sphere, and each ray mosaic $M_{j,n}$ represents the corresponding annular tangle $A_{j,n}$ with matching boundary data. Gluing these piecewise ambient homeomorphisms along the matched boundary connection points gives an ambient homeomorphism from $|M|$ to $K$; the closure in Definition~\ref{tree-mosaic-realization} adds exactly the limiting points corresponding to the $p_j$.

Each wild point $p_j$ corresponds to the unique infinite path $r \to v_{j,1} \to v_{j,2} \to \cdots$ in $T$, confirming that wild points are represented by combinatorial-wild points.
\end{proof}

\begin{theorem}\label{p-index-constraint}
    If $K$ is a knot with a point $p \in K$ that has $p$-index $P(K,p)>4$, there cannot exist a tree mosaic representative of $K$.
\end{theorem}

\begin{proof}
    Suppose, for contradiction, that $M$ is a tree mosaic representative of $K$. Since $P(K,p)>4$, the point $p$ is wild. Points in the interior of a finite tile, connecting segment, or finite-stage mosaic arc have tame neighborhoods, so $p$ must occur as the limit point associated to an infinite path $\gamma=(v_0,v_1,v_2,\ldots)$ in the tree.

    For each $j$, let $R_j$ be the geometric box occupied by the subtree rooted at $v_j$. These boxes are nested and have diameters tending to zero. Since the rest of the realization meets $R_j$ only through the boundary connection points inherited from the parent $T_\infty$ tile, we may choose an arbitrarily small regular neighborhood $U_j$ of $R_j$ whose boundary intersects $K$ only in small transverse cuts near those connection points. Since every $T_\infty$ tile in an $\rv{2,4}$ mosaic has either $2$ or $4$ connection points, we have
    \[
        |K \cap \partial U_j| \leq 4
    \]
    for arbitrarily small neighborhoods $U_j$ of $p$. Hence $P(K,p) \leq 4$, contradicting the assumption that $P(K,p)>4$.
\end{proof}

\begin{theorem}[Contracting an Infinite Ray]\label{inf-ray-contract}
    Let $M$ be a subtree mosaic whose tree is an infinite ray. Suppose that this ray $\gamma$, which is a combinatorial-wild point by definition, does not correspond to a wild point $p_\gamma$. Then, provided that $|M|$ is a tangle, there exists a finite subtree mosaic $N$ such that $|M|$ and $|N|$ are ambient homeomorphic.
\end{theorem}

\begin{proof}
    Let $\gamma$ denote the ray emanating from the root, and let $p_\gamma$ be the corresponding limit point from Proposition \ref{geometric-realization-prop}. Suppose that $p_\gamma$ is not a wild point.

    Since the tree of $M$ is an infinite ray, $\gamma$ is its only infinite path. Every point of $|M|$ other than $p_\gamma$ lies in a finite-stage tile arc, connecting segment, or finite union of such pieces, and therefore has a tame neighborhood. Thus the only possible wild point of $|M|$ is $p_\gamma$. By hypothesis, $p_\gamma$ is not wild, so $|M|$ is tame.

    Then, by Proposition \ref{tangle-rep}, there exists a tangle mosaic representative $L \in \LL^{(n)}$ for $|M|$.

    Let $N$ be a finite subtree mosaic on a single vertex, with the tangle mosaic $L$ associated to that vertex. Then $|M|$ and $|N|$ are ambient homeomorphic.
\end{proof}




\begin{remark}
    We may neglect combinatorial-wild points of a tree knot mosaic that do not correspond to true wild points on its realization. That is, without loss of generality, \textbf{we may assume that all combinatorial-wild points represent true wild points.}
\end{remark}


\begin{example} In Figure \ref{combo}, we show an example of a combinatorial-wild point that does not correspond to a wild point. Topologically, the pictured knot is the unknot.

\begin{figure}[H]
\centering
\includegraphics[width=.15\linewidth]{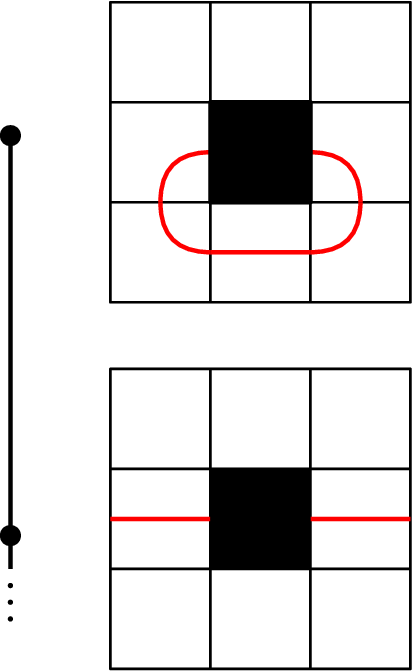} 
\caption{Ray-repeating mosaics always have one combinatorial-wild point, by construction.}\label{combo}
\end{figure}

\end{example}

\begin{remark}
    Notably, just because a knot has non-isolated wild points, that does not prohibit us from representing it as a mosaic. For example, consider a tree mosaic defined on an infinite binary tree. This has uncountably many combinatorial-wild points (i.e. infinitely many rays emanating from the root).

    \textbf{From now on, all tree mosaics we consider will be assumed to have isolated, and thus finitely many, wild points.}
\end{remark}

\subsection{Contractions}

\begin{definition}[Tree mosaic contractions]\label{def:contraction}
Let $\mathcal{M} = (T,F,i)$ be a tree mosaic. A \textbf{contraction} $\mathcal{M} \searrow \mathcal{M}'$ to a tree mosaic $\mathcal{M}' = (T',F',i')$ is defined as follows.

Let $\{S_j\}_{j \in J}$ be a countable collection of pairwise vertex-disjoint finite subtrees of $T$, no two of which are joined by an edge of $T$, each with root $r_j$ (the vertex of $S_j$ closest to the root of $T$). More general finite collections may be contracted successively.
\begin{itemize}
	\item \textbf{Tree contraction:} $T'$ is obtained from $T$ by contracting each $S_j$ to a single vertex $v_{S_j}$. Vertices not in any $S_j$ are unchanged.

	\item \textbf{Mosaic assignments:} For $u \notin \bigcup_{j} V(S_j)$, set $F'(u) = F(u)$. For a contracted vertex $v_{S_j}$, define $F'(v_{S_j})$ by induction on $|V(S_j)|$. 
    \begin{itemize}
        \item If $|V(S_j)| = 1$, set $F'(v_{S_j}) = F(r_j)$.
        \item If $|V(S_j)| > 1$, let $w_1, \ldots, w_m$ be the children of $r_j$ in $S_j$, and let $M_k$ be the mosaic obtained by (inductively) contracting the subtree of $S_j$ rooted at $w_k$. Then $F'(v_{S_j})$ is obtained from $F(r_j)$ by inserting $M_k$ into the $T_\infty$ tile targeted by $i(r_j, w_k)$, for each $k = 1, \ldots, m$.
    \end{itemize}
    
	\item \textbf{Edge embeddings:} For an edge $(a, b) \in E(T')$: 
	\begin{itemize}
		\item If neither $a$ nor $b$ is a contracted vertex, then $i'(a, b) = i(a, b)$.
        \item If $a = v_{S_j}$ and $b \notin V(S_j)$, then $b$ was a child of some $\ell \in V(S_j)$ in $T$. Since the contraction only replaces $T_\infty$ tiles whose corresponding children lie in $S_j$, the tile of $F(\ell)$ targeted by $i(\ell, b)$ survives in $F'(v_{S_j})$, and $i'(v_{S_j}, b)$ is the embedding into this tile.
        \item If $b = v_{S_j}$ and $a \notin V(S_j)$, then $a$ was the parent of $r_j$ in $T$. We set $i'(a, v_{S_j}) = i(a, r_j)$.
	\end{itemize}
\end{itemize}
\end{definition}

\begin{remark}
Intuitively, a contraction collapses a finite portion of the tree mosaic into a single vertex by ``realizing'' the embeddings throughout the subtree. (See Figure \ref{contraction-example}.) The resulting mosaic $F'(v_{S_j})$ is the bottom-up composition of all mosaics in the subtree, and the contracted tree mosaic represents the same knot. In particular, contracting an entire finite tree mosaic to a single vertex recovers the classical mosaic of the corresponding tame knot.
\end{remark}

\begin{figure}[H] 
\includegraphics[width=.5\linewidth]{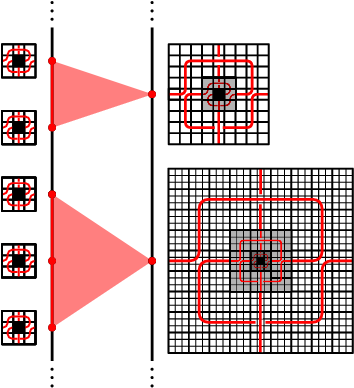}
\caption{A contraction of a tree mosaic.}\label{contraction-example}
\end{figure}

\begin{theorem}\label{thm:contraction-preserves-realization}
    A contraction of a tree mosaic preserves the represented knot type: if $\mathcal{M} \searrow \mathcal{M}'$, then $|\mathcal{M}|$ and $|\mathcal{M}'|$ are ambient homeomorphic.
\end{theorem}

\begin{proof}
Since the subtrees $\{S_j\}$ are pairwise vertex-disjoint and not joined by edges, their contractions occur in disjoint regions; it is enough to prove the claim for a single subtree $S$. We induct on $|V(S)|$; the case $|V(S)| = 1$ is trivial. For $|V(S)| > 1$, let $\ell$ be a leaf of $S$ with parent $p$ in $S$. The embedding $i(p, \ell)$ inserts $F(\ell)$ into the $T_\infty$ tile of $F(p)$, producing a collapsed mosaic at $p$. In $|\mathcal{M}|$, this $T_\infty$ tile contributes $\varnothing$ to $|F(p)|_{\text{tame}}$ while $\ell$ fills the region with $\sigma(\ell) \cdot |F(\ell)|_{\text{tame}} + \rho(\ell)$. The fundamental embedding realizes the same replacement inside the parent cube after a finite affine subdivision of that cube; the added zoom tiles are blank or straight connector tiles and introduce no new knotting. Thus absorbing $\ell$ into $p$ changes the realization by an ambient homeomorphism supported in the parent region (or, if $p$ is the root, by the same local homeomorphism together with a global affine rescaling of the containing box). Applying the inductive hypothesis to $S \setminus \{\ell\}$ gives the result.
\end{proof}


\begin{definition}
    A tree mosaic $\mathcal{M} = (T, F, i)$ is \textbf{star-like} if its underlying tree $T$ consists of a root with $N \geq 1$ children, each the start of an infinite ray---that is, every non-root vertex has exactly one child.
\end{definition}

\begin{theorem}[K\"onig's Lemma\footnote{The classic way to remember this statement is: ``If the human race never dies out, somebody now living has a line of descendants that will never die out."}]
    Every infinite, locally finite, rooted tree contains an infinite path starting from the root.
\end{theorem}

\begin{lemma}\label{finite-subtrees}
    Let $\mathcal{M} = (T, F, i)$ be a tree mosaic with finitely many combinatorial-wild points $\gamma_1, \ldots, \gamma_N$. Every vertex of $T$ not lying on any $\gamma_k$ roots a finite subtree.
\end{lemma}

\begin{proof}
    Suppose $v \in V(T)$ does not lie on any $\gamma_k$ and roots an infinite subtree $T_v$. Since $T$ is locally finite, K\"onig's lemma yields an infinite path in $T_v$ starting at $v$. Concatenating with the unique path from the root to $v$ produces a combinatorial-wild point not among $\gamma_1, \ldots, \gamma_N$, a contradiction.
\end{proof}

\begin{theorem}\label{star-thm}
    If a tree mosaic $\mathcal{M}$ has finitely many combinatorial-wild points, then it can be contracted to a star-like tree.
\end{theorem}

\begin{proof}
    Let $\gamma_1, \ldots, \gamma_N$ denote the combinatorial-wild points of $\mathcal{M} = (T, F, i)$, and let $S = \bigcup_{k=1}^N V(\gamma_k)$ be the set of vertices lying on at least one $\gamma_k$. By Lemma~\ref{finite-subtrees}, every vertex not in $S$ roots a finite subtree of $T$.

    The $N$ paths all originate at the root $r$ and pairwise diverge at finitely many vertices---at most $\binom{N}{2}$ divergence points, each at finite depth. Let $D$ denote the maximum depth of any divergence vertex. Beyond depth $D$, each $\gamma_k$ occupies vertices lying on no other $\gamma_j$, so the subgraph induced by $S$ decomposes as a finite core (vertices at depth $\leq D$) together with $N$ disjoint infinite rays.

    We apply contractions in two stages. First, iteratively contract every maximal finite subtree of $T \setminus S$ into the vertex of $S$ to which it is attached, composing the embeddings along each branch. After this step, $V(T) = S$ and each vertex on a ray beyond depth $D$ has exactly one child. Second, contract the remaining finite core to a single root vertex by composing the embeddings along each path in the core. The result is a star-like tree mosaic: a single root with $N$ infinite rays. See an example in Figure \ref{star-like}.
\end{proof}

\begin{figure}[h]    \includegraphics[width=.8\linewidth]{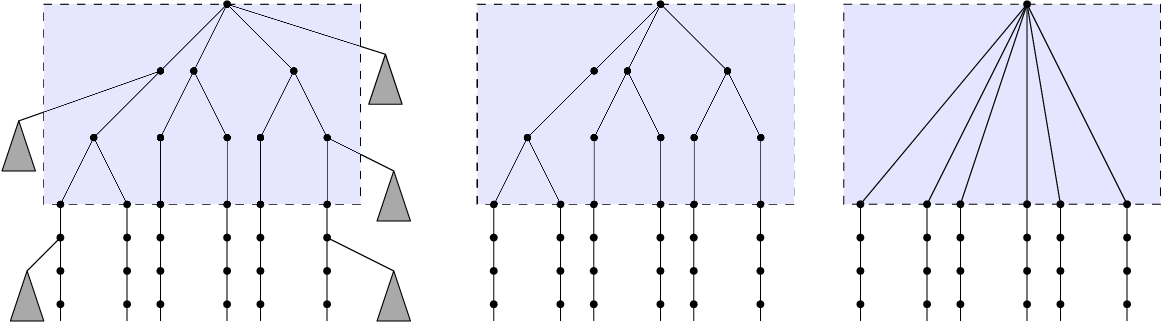} 
    \caption{According to the proof of Theorem \ref{star-thm}, we apply two contractions to obtain a star-like tree mosaic from an arbitrary one with finitely many combinatorial-wild points. The figures, moving from left to center, show the contraction of all finite portions of the tree to a core set of vertices. Moving from center to right, we picture the contraction of a finite core set of vertices up to a depth $D$ to the root vertex.}\label{star-like}
\end{figure}

\subsection{Equivalences for Tree Mosaics (with Isolated Wild Points)}

In this section, we will describe a number of types of equivalences that can be applied to a tree mosaic $M$ with tree $T=(V,E)$ and maps $F$ and $i$ of increasing complexity. As a first step, we define a {\em local equivalence} via a replacement of a single mosaic assignment.

\begin{remark}[Isolated wild points assumption]\label{rem:isolated-wild-points}
    Although the tree mosaic definition (Definition~\ref{treemosaic}) allows for non-isolated wild points---for instance, a tree mosaic whose underlying tree has uncountably many infinite rays converging to a common accumulation point---the equivalence theory developed in this section applies only to tree mosaics whose wild points are \emph{isolated}. 
\end{remark}

We now define tree mosaic versions of moves $V^*$, $VIII.a$, and $VIII.b$ from
Figure~\ref{rv-mosaic-reid}. These come in two fundamentally different types:
$V^*$ is an \emph{orientation change} of a child embedding, while $VIII.a$ and
$VIII.b$ are \emph{local tame moves} in the parent vertex mosaic.

\begin{definition}[Tree Mosaic Move $V^*$]\label{def:tree-move-Vstar}
  Let $\mathcal{M} = (T, F, i)$ be a suitably connected tree mosaic, let
  $v \in V(T)$ be a non-root vertex with parent $u$, and let $\sigma \in D_4$.
  The \textbf{tree mosaic move $V^*$ at $(u,v)$ by $\sigma$} produces a new tree
  mosaic $\mathcal{M}' = (T, F', i')$ where:
  \begin{itemize}
    \item $F'(w) = F(w)$ for all $w \in V(T)$ (vertex mosaics unchanged), and
    \item for the edge $(u, v)$, the embedding is replaced by the
      $\sigma$-oriented embedding:
      \[
        i'(u, v) = i(u, v)^\sigma,
      \]
      where $x_{ab}^\sigma(M, N) = x_{ab}(\sigma \cdot M, N)$ as in
      Definition~\ref{def:oriented-embed}; all other edge embeddings are unchanged.
  \end{itemize}
  Informally, move $V^*$ at $(u,v)$ by $\sigma$ re-embeds the child mosaic $F(v)$
  into the $T_\infty$ tile of $F(u)$ in the $\sigma$-rotated/reflected orientation,
  without altering any vertex mosaic or any other embedding.
\end{definition}

\begin{remark}
  At the level of geometric realizations, move $V^*$ at $(u,v)$ by $\sigma$ replaces
  the embedding of $|F(v)|_{\mathrm{tame}}$ inside the $T_\infty$ region of $|F(u)|$
  with the embedding of $|\sigma \cdot F(v)|_{\mathrm{tame}}$ in the same region.
  Since $\sigma \in D_4$ acts by an isometry of the ambient cube (Definition~\ref{def:D4-action}),
  the resulting realizations are ambient isotopic. This is the tree mosaic analogue of the
  classical RV move $V^*$, which flips or rotates a rigid vertex together with
  its incident arcs.
\end{remark}

\begin{definition}[Tree Mosaic Moves $VIII.a$ and $VIII.b$]\label{def:tree-move-VIII}
  Let $\mathcal{M} = (T, F, i)$ be a suitably connected tree mosaic, and let
  $v \in V(T)$. Suppose the vertex mosaic $F(v)$ contains a $T_\infty$ tile at
  position $(a, b)$, with a tame strand of $F(v)$ routed adjacent to this tile on
  a specified side.

  The \textbf{tree mosaic move $VIII.a$ (resp.\ $VIII.b$) at vertex $v$} replaces
  $F(v)$ with a new mosaic $F'(v)$ obtained by applying the $\rv{2,4}$ mosaic
  move $VIII.a$ (resp.\ $VIII.b$) from Figure~\ref{rv-mosaic-reid} locally within
  $F(v)$, leaving the $T_\infty$ tile at $(a,b)$ and all tiles outside the local
  neighborhood of $(a,b)$ unchanged. All other vertex mosaics and all embedding
  functions are unchanged:
  \begin{itemize}
    \item $F'(w) = F(w)$ for all $w \neq v$,
    \item $i'(e) = i(e)$ for all edges $e \in E(T)$.
  \end{itemize}
\end{definition}

\begin{remark}
  Tree mosaic $V^*$ moves correspond to changes in the embedding map $i$, while
  tree mosaic $VIII.a$/$VIII.b$ moves correspond to changes in the vertex mosaic map $F$.
  Together with mosaic-level equivalence moves (planar isotopy, Reidemeister I--III,
  and rigid-vertex moves IV and VI applied to individual vertex mosaics) and
  contractions (Definition~\ref{def:contraction}), these generate the full set of
  tree mosaic equivalence operations.
\end{remark}

\begin{definition}[Mosaic-Level Equivalence]\label{def:mosaic-level-equivalence}
    Two well-defined, suitably connected tree knot mosaics $M_1 = (T_1,F_1,i_1)$ and $M_2 = (T_2,F_2,i_2)$ are \textbf{mosaic-level equivalent} if
    \begin{enumerate}
        \item $T_1 = T_2$ and $i_1 = i_2$, and
        \item for each vertex $v \in V(T_1) = V(T_2)$, the mosaics $F_1(v)$ and $F_2(v)$
        have the same dimension, have $T_\infty$ tiles in the same positions with the same connection data for every outgoing embedding, and are related by a finite sequence of $\rv{2,4}$ mosaic Reidemeister moves preserving those $T_\infty$ positions and connection data.
    \end{enumerate}
\end{definition}


\subsection{Connection to Kobayashi Equivalence}

Performing Reidemeister moves on individual vertex mosaics at increasing
depth in a tree mosaic corresponds, at the level of geometric realizations,
to applying a sequence of ambient homeomorphisms whose supports shrink to
zero diameter.
Theorem~\ref{thm:vks-ambient-homeomorphism} below, due to Kobayashi
\cite{forest}, formalizes the conditions under which such a countable
composition converges to an ambient homeomorphism of $\RR^3$.
We introduce the following terminology to organize these conditions.

\begin{definition}[Kobayashi sequence]\label{def:kobayashi-sequence}
  A \textbf{Kobayashi sequence} is a pair
  \[
    ((h_k)_{k \in \NN}, (V_k)_{k \in \NN})
  \]
  where, for each $k \in \NN$, the map
  $h_k \colon \RR^3 \to \RR^3$ is a homeomorphism with
  $h_k|_{V_k^c} = \mathrm{id}$, and there exists a compact set
  $A \subseteq \RR^3$ such that:
  \begin{enumerate}
    \item \textbf{(Nested supports)}
      $\cdots \subseteq V_{k+1} \subseteq V_k \subseteq \cdots
      \subseteq V_1 \subseteq A^\circ$,
      and
    \item \textbf{(Vanishing diameters)}
      $\displaystyle\lim_{k \to \infty}
        \operatorname{diam}(V_k) = 0$.
  \end{enumerate}
  We call $V_k$ the \textbf{support} of $h_k$ in this context, and we
  write $\ms{h}_n = \comp_{k=1}^n h_k = h_n \circ \cdots \circ h_1$ for
  the partial composition.
\end{definition}

\begin{remark}\label{remark:kobayashi-gives-homeomorphism}
  Conditions (1) and (2) of Definition~\ref{def:kobayashi-sequence}
  are structural properties of the sequence that can be verified
  individually for each $h_k$ and $V_k$.
  A third condition is required to apply Theorem~\ref{thm:vks-ambient-homeomorphism}:
  \begin{enumerate}
    \setcounter{enumi}{2}
    \item $\ms{h}_\infty := \lim_{n \to \infty} \ms{h}_n$
      exists and is bijective.
  \end{enumerate}
  When conditions (1)--(3) all hold, Theorem~\ref{thm:vks-ambient-homeomorphism}
  guarantees that $\ms{h}_\infty$ is an ambient homeomorphism of $\RR^3$.
  Conditions (1) and (2) are often immediate from the geometry of the
  construction; condition (3) must be verified separately in each
  application, and is in general the most delicate requirement.
\end{remark}

\begin{remark}\label{remark:tree-mosaic-kobayashi}
  In the tree mosaic setting, the Kobayashi framework arises naturally.
  When a Reidemeister move is applied to the vertex mosaic $F(v)$ at a
  vertex $v$ of depth $d$, the resulting homeomorphism $h_v$ of $\RR^3$
  can be supported within the geometric realization of the subtree rooted at
  $v$.
  Since scale factors decrease geometrically with depth
  (cf.\ Definition~\ref{tree-mosaic-realization}), the diameters of these supports
  tend to zero as $d \to \infty$.
  Along a single branch these supports are nested, while supports belonging
  to incomparable branches have disjoint interiors and meet, at worst, along
  boundary faces where the relevant maps are fixed. Thus the global tree
  situation is a branchwise/nested and off-branch/disjoint variant of the
  Kobayashi setup rather than a literal single nested sequence in general.
  In Lemma~\ref{lem:mosaic-level-equiv} below, we use this geometry directly
  by patching shell-supported homeomorphisms.
\end{remark}

\begin{remark}\label{remark:kobayashi-sufficient-not-equivalence}
  Given a Kobayashi-type sequence satisfying conditions (1)--(3), the
  resulting ambient homeomorphism $\ms{h}_\infty$ takes any knot $K_1$ in
  its domain to an ambient-homeomorphic knot $K_2 = \ms{h}_\infty(K_1)$.
  One can therefore say that $K_1$ and $K_2$ are \textbf{related by a
  Kobayashi-type homeomorphism}.
  This relation is reflexive: every knot is related to itself via the
  constant sequence $h_k = \mathrm{id}$ with $V_k = \varnothing$.
  However, we do not claim that this relation is symmetric or transitive
  in general.
  For our purposes, the forward direction suffices: the proofs below
  construct explicit Kobayashi-type sequences to exhibit ambient
  homeomorphisms between the geometric realizations of equivalent tree
  mosaics.
  We note that the specific patching constructions used below have supports
  with disjoint interiors, and the associated homeomorphisms agree as the
  identity on shared boundaries. In this setting the order of composition is
  immaterial, and the inverse is obtained by applying the inverse local maps
  on the same supports; see the proof of Lemma~\ref{lem:mosaic-level-equiv}
  for details.
\end{remark}

\begin{theorem}[\cite{forest}, Theorem 3.3 \& Remark 3.5] \label{thm:vks-ambient-homeomorphism}
  For all $k \in \NN$, let $h_k
  \colon \RR^3 \to \RR^3$ be a homeomorphism, and for all $n \in \NN$, define
  \[
    \ms h_n = \comp_{k=1}^n h_k = (h_n \circ h_{n-1} \circ \cdots
    \circ h_{2} \circ h_1).
  \]
  For each $k$ let $V_k \subseteq \RR^3$ such that $h_k$ is identity on
  $V_k^c$. Suppose there exists a compact $A \subseteq \RR^3$ such that
  \begin{enumerate}
    \item \textbf{(Nested supports)}
      $\cdots \subseteq V_{k+1} \subseteq V_k \subseteq \cdots
      \subseteq V_1 \subseteq A^\circ$,
    \item \textbf{(Vanishing diameters)}
      $\displaystyle\lim_{k\to\infty} \operatorname{diam}(V_k) = 0$,
      and
    \item $\ms h_\infty := \displaystyle\lim_{n\to\infty} \ms h_n$
      exists and is bijective.
  \end{enumerate}
  Then $\ms h_\infty$ is a homeomorphism.
\end{theorem}

Note that the subsets $V_k$ correspond nicely with geometric realizations of tree mosaic branches. The theorem motivates the equivalence arguments below, where the relevant local homeomorphisms are supported on shrinking tree-mosaic regions and are then patched directly.

\begin{proposition}[$T_\infty$-identity property]\label{prop:tinfty-identity}
    Let $L_1, L_2$ be $\rv{2,4}$ tangle $n$-mosaics related by a finite sequence of $\rv{2,4}$ mosaic Reidemeister moves, with $T_\infty$ tiles at the same positions $(a_1, b_1), \ldots, (a_m, b_m)$ and the same connection data on the faces of each $T_\infty$ tile (i.e., the adjacent tame tiles have the same connection points facing each $T_\infty$ tile in both $L_1$ and $L_2$).
    Then the ambient isotopy $\tilde{h} \colon [0,n]^2 \times [0,1] \to [0,n]^2 \times [0,1]$ guaranteed by Proposition~\ref{prop:lk-easy-rv} can be chosen so that $\tilde{h}$ is the identity on the entirety of each $T_\infty$ sub-cube $C_j := [a_j - 1, a_j] \times [b_j - 1, b_j] \times [0,1]$.
\end{proposition}

\begin{proof}
    It suffices to show the claim for a single Reidemeister move; the general case follows by composition (each step fixes every~$C_j$, so the composition does as well). Each $\rv{2,4}$ mosaic Reidemeister move modifies a set of tame tiles at specified positions; the $T_\infty$ tile at $(a_j, b_j)$ is \emph{never among the modified tiles}: standard moves (planar isotopy, R1--R3) involve only tame tiles, and rigid vertex moves (IV, V$^*$, VI) rearrange arcs in the tame tiles adjacent to the rigid vertex while the $T_\infty$ tile itself is unchanged (its mosaic entry is not altered by any move).

    Let $W$ denote the union of the (closed) modified tile sub-cubes (for rigid vertex moves, $W$ consists of the adjacent tame tiles only---not~$C_j$, since the $T_\infty$ tile is unmodified). The arc configurations before and after the move have identical boundary data on~$\partial W$: each move preserves connection points on every tile face. We construct the ambient isotopy \emph{tile by tile, rel\/~$\partial W$}. Within each modified tile sub-cube~$C$ (a PL $3$-ball), the old and new arc configurations are proper PL arcs sharing the same connection points on~$\partial C$, with arc interiors contained in~$C^\circ$. Since all tame tile arcs are unknotted proper PL arcs in a $3$-ball, any two such arc systems with the same boundary data are PL isotopic rel~$\partial C$ \cite[Corollary~3.23 and \S4]{rourkesanderson1972}; by the isotopy extension theorem \cite[Theorem~4.11]{rourkesanderson1972}, this rel-boundary arc isotopy extends to an ambient PL isotopy of~$C$ that is stationary on~$\partial C$, hence fixes~$\partial C$ pointwise. Define~$\tilde{h}$ to be this isotopy on each modified tile and the identity elsewhere; the pieces agree on shared faces by the rel-boundary construction (and by the pasting lemma), giving a well-defined continuous ambient isotopy.

    Since no $T_\infty$ tile is among the modified tile sub-cubes, the open interior $C_j^\circ$ is disjoint from~$W^\circ$ (distinct tile sub-cubes have disjoint interiors). Every face of $C_j$ either lies on~$\partial W$ (if shared with a modified tile), on~$\partial([0,n]^2 \times [0,1])$ (if~$C_j$ is a boundary tile), or in the complement of~$W$ entirely. In all cases $\tilde{h}$ is the identity: on~$\partial W$ by the rel-boundary construction, and outside~$W$ by the identity extension. Therefore $\tilde{h}$ fixes~$C_j$ pointwise.
\end{proof}

\begin{lemma}[Mosaic-level equivalence implies ambient homeomorphism]\label{lem:mosaic-level-equiv}
    If $M_1$ and $M_2$ are mosaic-level equivalent tree knot mosaics, then their geometric realizations $|M_1|$ and $|M_2|$ are ambient homeomorphic. More precisely, there exists an ambient homeomorphism $\ms{h}_\infty \colon \RR^3 \to \RR^3$ with $\ms{h}_\infty(|M_1|) = |M_2|$.
\end{lemma}

\begin{proof}
    We construct $\ms{h}_\infty$ by patching together vertex-level homeomorphisms on a family of pairwise-disjoint shells inside $\RR^3$.

    For each vertex $v$ of $T$ at depth $d$, define its \textbf{geometric region}
    $$R(v) := \rho(v) + [0,\; \sigma(v) \cdot \dim F(v)]^2 \times [0,\; \sigma(v)],$$
    where $\sigma(v), \rho(v)$ are the scale and position from Definition~\ref{tree-mosaic-realization}. This is the bounding box in which the realization of $F(v)$ sits: anchored at $\rho(v)$, with square footprint of side $\sigma(v) \cdot \dim F(v)$ in the $xy$-plane and height $\sigma(v)$ in the $z$-direction. For non-root vertices, using $\dim F(v) \geq 3$, one checks $\operatorname{diam}(R(v)) \leq \sqrt{3} \cdot 3^{-(d-1)}$; the root region is a fixed compact box. The regions satisfy:
    \begin{enumerate}
        \item[(i)] \textbf{Nested descendants.} If $u$ is a descendant of $v$, then $R(u) \subseteq R(v)$: each child $c$ embeds at a $T_\infty$ tile of $R(v)$, so $R(c)$ sits inside the corresponding tile sub-cube, and the property iterates.
        \item[(ii)] \textbf{Disjoint off-branch regions.} If $u, u'$ share no ancestor-descendant relationship, then $R(u)^\circ \cap R(u')^\circ = \varnothing$; for immediate siblings the regions do not even share a face, since black tiles are never adjacent.
    \end{enumerate}
    Define the \textbf{shell} of $v$ as
    $$V_v \;:=\; R(v) \setminus \bigcup_{c \text{ child of } v} R(c).$$
    Properties (i) and (ii) make the shells $\{V_v\}$ have pairwise disjoint interiors: if $v \neq v'$ are off-branch then $R(v)^\circ \cap R(v')^\circ = \varnothing$, and if $v'$ is a descendant of $v$ with $c$ the child of $v$ on the path to $v'$, then $V_{v'} \subseteq R(v') \subseteq R(c)$, which was excised from $V_v$. Possible intersections of shells occur only along shared boundary faces, edges, or corners. The shells, together with the exterior of $R(\text{root})$ and the wild points $p_\gamma = \bigcap_d R(v_d)$ (limits of infinite paths in $T$), decompose $\RR^3$ up to these shared boundary sets.

    Next, we construct a vertex-level homeomorphism for each $v$. The mosaics $F_1(v), F_2(v)$ are related by $\rv{2,4}$ Reidemeister moves with $T_\infty$ tiles at the same positions, and these moves preserve boundary strands, so (together with $i_1 = i_2$) the connection data at each $T_\infty$ tile agrees between $F_1(v)$ and $F_2(v)$. Proposition~\ref{prop:tinfty-identity} then yields a homeomorphism $\tilde{h}_v$ of the mosaic cube $[0, \dim F_1(v)]^2 \times [0,1]$ carrying $|F_1(v)|_{\mathrm{tame}}$ to $|F_2(v)|_{\mathrm{tame}}$, fixing the cube's boundary, and fixing every $T_\infty$ sub-cube pointwise. Rescaling $\tilde{h}_v$ to act on $R(v)$ and extending by the identity defines $h_v \colon \RR^3 \to \RR^3$. Because $h_v$ fixes $\partial R(v)$ and every child region $R(c)$ pointwise, it restricts to a self-homeomorphism of the shell $V_v$ that fixes $\partial V_v$, and is the identity outside $V_v$.

    We now patch these together: define $\ms{h}_\infty \colon \RR^3 \to \RR^3$ by
    $$\ms{h}_\infty(q) \;=\; \begin{cases} h_v(q), & q \in V_v \text{ for some vertex } v, \\ q, & \text{otherwise}. \end{cases}$$
    This is unambiguous: on any shared boundary set, all relevant maps are the identity. The inverse is given by the same recipe applied to $\{h_v^{-1}\}$: each $h_v^{-1}$ is a self-homeomorphism of $V_v$ fixing $\partial V_v$, so $\ms{h}_\infty^{-1}$ is well-defined and a genuine two-sided inverse. Hence $\ms{h}_\infty$ is a bijection of $\RR^3$.

    We check continuity in two cases, corresponding to whether the point in question is tame or wild. For a non-wild point $q$, either $q$ lies outside $R(\text{root})$ or $q$ belongs to some shell $V_v$. In either case, a neighborhood of $q$ is governed by a single map — the identity, or a particular $h_v$ — so continuity at $q$ is inherited from that map. At shell boundaries, every relevant vertex map is the identity (on $\partial R(v)$, $h_v$ is the identity by construction; on $\partial R(c)$, $h_v$ fixes the $T_\infty$ sub-cube $R(c)$ pointwise), so the local descriptions glue continuously.

    Now consider a wild point $p_\gamma$, where $\gamma = (v_0, v_1, \ldots)$. The key observation is that $R(v_D)$ shrinks down onto $p_\gamma$ as $D \to \infty$, and $\ms{h}_\infty$ preserves each $R(v_D)$, so nearby points stay nearby.

    Given $\varepsilon > 0$, choose $D$ with $\operatorname{diam}(R(v_D)) < \varepsilon$. Every $h_v$ sends $R(v_D)$ into itself: strict ancestors fix $R(v_D)$ pointwise (Proposition~\ref{prop:tinfty-identity}), $v_D$ and its descendants are supported in $R(v) \subseteq R(v_D)$, and off-branch $v$ have $R(v)^\circ$ disjoint from $R(v_D)^\circ$. Hence $\ms{h}_\infty(R(v_D)) \subseteq R(v_D)$. Moreover, by Definition~\ref{rv24def}, $T_\infty$ tiles are forbidden on the boundary of a tangle mosaic, so $R(v_{D+1})$ sits at an interior tile position of $F(v_D)$, a positive distance from the \emph{lateral boundary} of $R(v_D)$ (i.e.\ $\partial R(v_D) \setminus (\{z = 0\} \cup \{z = \sigma(v_D)\})$):
    $$\operatorname{dist}(R(v_{D+1}),\, \partial_{\mathrm{lateral}} R(v_D)) \;\geq\; \sigma(v_D) \;>\; 0.$$
    In the $z$-direction, $p_\gamma$ sits on the bottom face $\{z = 0\}$ shared by every $R(v) \subseteq \{z \geq 0\}$. Setting $\delta = \min\bigl(\varepsilon,\; \tfrac{1}{2}\operatorname{dist}(p_\gamma, \partial_{\mathrm{lateral}} R(v_D)),\; \sigma(v_D)\bigr)$, any $q \in B(p_\gamma, \delta)$ with $q_z \geq 0$ lies in $R(v_D)$, so $\|\ms{h}_\infty(q) - p_\gamma\| \leq \operatorname{diam}(R(v_D)) < \varepsilon$; any $q$ with $q_z < 0$ lies below every region, is fixed by $\ms{h}_\infty$, and satisfies $\|\ms{h}_\infty(q) - p_\gamma\| < \delta \leq \varepsilon$. Thus $\ms{h}_\infty(B(p_\gamma, \delta)) \subseteq B(p_\gamma, \varepsilon)$. The same argument applied to $\ms{h}_\infty^{-1}$ gives continuity there as well, so $\ms{h}_\infty$ is a homeomorphism of $\RR^3$.

    Finally, we show that $\ms{h}_\infty(|M_1|) = |M_2|$. For each vertex $v$, $h_v$ carries $\sigma(v) \cdot |F_1(v)|_{\mathrm{tame}} + \rho(v)$ to $\sigma(v) \cdot |F_2(v)|_{\mathrm{tame}} + \rho(v)$, and no other vertex map disturbs this piece. For each parent-child edge $(u,v)$, the connecting segment $\mathcal{C}(u,v) \subseteq R(u)$ is carried by $h_u$ to the corresponding segment for $M_2$: the parent endpoint moves by the prescribed $\rv{2,4}$ isotopy, while the child endpoint lies in a $T_\infty$ sub-cube and is fixed. Wild points are fixed by $\ms{h}_\infty$ and coincide across $M_1, M_2$ since the tree, scale, position, and embedding data agree. Since $\ms{h}_\infty$ is a homeomorphism, it commutes with closure, giving $\ms{h}_\infty(|M_1|) = |M_2|$.
\end{proof}

\begin{definition}[Tree Mosaic Equivalence]
    Let $M, P, P', N$ be tree knot mosaics. Suppose $P, P'$ are mosaic-level equivalent, and let $P$ be a contraction of $M,$ and $P'$ a contraction of $N$. Then, $M$ and $N$ are \textbf{equivalent}.
\end{definition}

\begin{corollary}[Tree Mosaic Equivalence Implies Topological Equivalence]\label{cor:tree-mosaic-equiv}
    If $M$ and $N$ are equivalent tree knot mosaics, then $|M|$ and $|N|$ are ambient homeomorphic.
\end{corollary}

\begin{proof}
    By definition, there exist tree knot mosaics $P, P'$ such that $P$ is a contraction of $M$, $P'$ is a contraction of $N$, and $P, P'$ are mosaic-level equivalent.
    By Theorem~\ref{thm:contraction-preserves-realization}, contractions preserve the represented knot type, so $|M|$ is ambient homeomorphic to $|P|$ and $|N|$ is ambient homeomorphic to $|P'|$.
    By Lemma~\ref{lem:mosaic-level-equiv}, $|P|$ and $|P'|$ are ambient homeomorphic.
    Therefore $|M|$ and $|N|$ are ambient homeomorphic.
\end{proof}




\section{Future Directions}

\begin{question}
    Which invariants from tame knot theory can be generalized to the wild setting by taking advantage of the tree mosaic structure provided in this paper?
\end{question}

\begin{question}
    What other types of wild knots can we capture by generalizing the definition of $\rv{2,4}$ $n$-mosaics to rigid vertices with larger degrees, allowing for adjacent $T_{\infty}$ tiles and insertions of $m$-tangles where $m>2$? (See Figure \ref{big_ones}.)
\end{question}



\begin{figure}[h]
\centering
\includegraphics[width=.6\linewidth]{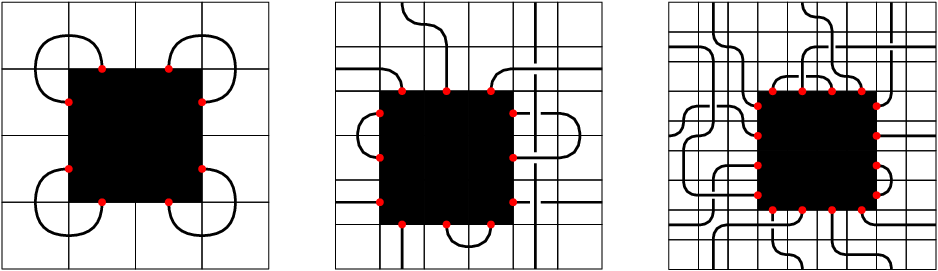}
\caption{Examples of rigid vertex knot and tangle mosaics with larger regions of adjacent $T_{\infty}$ tiles.}\label{big_ones}
\end{figure}


\newpage

\printbibliography

\end{document}